\documentclass[journal]{IEEEtran}
\usepackage{amsthm}
\usepackage{amsmath,amssymb,amsfonts}
\allowdisplaybreaks
\newcommand\numberthis{\addtocounter{equation}{1}\tag{\theequation}}

\usepackage{graphicx}
\usepackage{graphics}
\usepackage{subfigure}

\usepackage{float}

\usepackage{varioref}
\usepackage{hyperref}
\usepackage[noabbrev]{cleveref}

\usepackage{xcolor}

\usepackage{caption}

\usepackage{cite}

\ifCLASSINFOpdf
\else
\fi
\usepackage{algorithm}
\usepackage{algorithmic}
\newtheorem{definition}{Definition}
\newtheorem{theorem}{Theorem}
\newtheorem{lemma}{Lemma}
\newtheorem{corollary}{Corollary}
\newtheorem{assumption}{Assumption}
\newtheorem{remark}{Remark}

\begin{document}
%
\title{Input-Feedforward-Passivity-Based Distributed Optimization Over Jointly Connected Balanced Digraphs}
%
%
%

\author{Mengmou~Li,~
        Graziano~Chesi,~
        and~Yiguang~Hong
\thanks{A preliminary version of this work was presented in the 58th IEEE Conference on Decision and Control, Nice, France \cite{li2019input}.}
\thanks{The work of Y.~Hong was supported by National Natural Science Foundation of China under Grant 61733018.}
\thanks{M.~Li and G.~Chesi are with the Department
of Electrical and Electronic Engineering, The University of Hong Kong, Hong Kong, China (e-mail: mmli.research@gmail.com; chesi@eee.hku.hk).}
\thanks{Y.~Hong is with the Key Laboratory of Systems and
Control, Academy of Mathematics and Systems Science, Chinese Academy of Sciences, 100190, Beijing, China (e-mail: yghong@iss.ac.cn).}

}

\maketitle
\begin{abstract}
In this paper, a distributed optimization problem is investigated via input feedforward passivity.
First, an input-feedforward-passivity-based continuous-time distributed algorithm is proposed. It is shown that the error system of the proposed algorithm can be decomposed into a group of individual input feedforward passive (IFP) systems that interact with each other using output feedback information.
Based on this IFP framework, convergence conditions of a suitable coupling gain are derived over weight-balanced and uniformly jointly strongly connected (UJSC) topologies.
It is also shown that the IFP-based algorithm converges exponentially when the topology is strongly connected.
Second, a novel distributed derivative feedback algorithm is proposed based on the passivation of IFP systems.
While most works on directed topologies require knowledge of eigenvalues of the graph Laplacian, the derivative feedback algorithm is fully distributed, namely, it is robust against randomly changing weight-balanced digraphs with any positive coupling gain and without knowing any global information.
Finally, numerical examples are presented to illustrate the proposed distributed algorithms.
\end{abstract}
%

%
\IEEEpeerreviewmaketitle

\section{Introduction}
%
%
%
%
\IEEEPARstart{D}{istributed} optimization over multi-agent systems has been widely investigated in recent years, due to its broad applications in various aspects including wireless networks, smart grids, and machine learning. In addition to discrete-time algorithms (e.g., \cite{nedic2015distributed,nedic2017achieving,xie2018distributed}),
a variety of continuous-time distributed algorithms have been proposed to solve distributed optimization problems \cite{wang2010control,yi2015distributed,li2018generalized,zeng2018distributedb}.
Continuous-time algorithms can be implemented in hardware devices like analog circuits \cite{forti2004generalized}, and achieve tasks such as motion coordination of multi-agent systems \cite{martinez2007motion}.
Studying optimization in the continuous-time domain benefits from numerous control techniques for stability analysis and also opens up the possibility to address commonly encountered problems in large-scale networks, such as disturbance rejection \cite{wang2015distributed}, robustness to delays or uncertainties \cite{zhang2017distributed,hatanaka2018passivity}, or channel constraints \cite{kia2015distributed}.
However, most of the proposed algorithms are only for undirected topologies and not applicable to directed topologies \cite{wang2010control,yi2015distributed,li2018generalized,zeng2018distributedb}. To deal with this difficulty, some parameters in the algorithms can be tuned to stabilize the dynamics \cite{gharesifard2014distributed, li2017distributed}, while some variants of the standard proportional-integral algorithm are proposed \cite{kia2015distributed,deng2017distributed}. However, most of these methods often employ coordinate transformation along with complicated Lyapunov function candidates in convergence analysis, which does not preserve network structures, and requires eigenvalues of the graph Laplacian to design some parameters \cite{kia2015distributed,gharesifard2014distributed,deng2017distributed,zhu2018continuous}. Compared with these methods, a systematic approach that focuses more on the distributed interconnection of agents in the network is needed.

It is well known that dissipativity (as well as its special case, passivity) is a useful tool for stability analysis and control design \cite{chopra2006passivity,chopra2012output,liu2015output}. Recently, there emerged some passivity-based algorithms on distributed optimization under some communication constraints \cite{stegink2017unifying, tang2017distributed, hatanaka2018passivity, yamashita2020passivity}. However, these passivity-based algorithms can only be applied over undirected graphs, while it is shown that output consensus can be achieved over directed graphs through simple output feedback interconnections of passive systems \cite{chopra2006passivity,chopra2012output}. Motivated by these works, we aim to study distributed algorithms over directed graphs via passivity techniques.
On one hand, we conjecture that it is in general difficult to directly construct a distributed algorithm that can be interpreted as output feedback interconnections of passive systems.
On the other hand, works in \cite{qu2014modularized, proskurnikov2017simple, li2019consensus} point out that output consensus can be achieved over directed graphs even among IFP (or passivity-short) systems. Therefore, if a distributed algorithm inherits input feedforward passivity, it can be directly applied to weight-balanced digraphs through output feedback interconnections.
As a byproduct of having the IFP properties, the distributed algorithm is also applicable over uniformly jointly strongly connected (UJSC) topologies.
This feature is remarkable since it greatly reduces communication costs, and hence is more practical in large-scale networks.
%
Though the problem of UJSC switching topologies has been considered in discrete-time algorithms \cite{nedic2015distributed,nedic2017achieving,xie2018distributed}, to the best of our knowledge, it has never been addressed in the continuous-time domain, due to the difficulties in stability analysis under the time-varying nature and lack of connectedness of topologies.

In this paper, we investigate the distributed optimization problem via input feedforward passivity.
First, we propose an IFP-based distributed algorithm whose error system is decomposed into a group of individual IFP systems that interact with each other using output feedback information.
Based on this IFP framework, we study the distributed algorithm over directed and UJSC weight-balanced topologies and derive convergence conditions of a suitable coupling gain for the algorithm.
We also show that this IFP-based algorithm converges exponentially when the graph is strongly connected.
Second, we propose a novel distributed derivative feedback algorithm based on the passivation of IFP systems.
While most works on directed topologies in the literature require knowledge of eigenvalues of the graph Laplacian \cite{kia2015distributed,gharesifard2014distributed,deng2017distributed,zhu2018continuous}, we show that the derivative feedback algorithm is fully distributed, namely, it is robust against randomly changing weight-balanced digraphs with any positive coupling gain and without knowing any global information. In other words, the derivative feedback algorithm is applicable over gossip-like balanced digraphs \cite{zhang2018distributed}, reducing communication costs. It is worth mentioning that \cite{li2017distributed} develops a fully distributed adaptive algorithm. However, it does not apply to switching or UJSC topologies.
The challenges in our work lie in the construction of a group of verifiable nonlinear IFP systems that solves the distributed optimization problem, the design of the fully distributed algorithm, and the convergence analysis of the proposed algorithms.

Moreover, our analytical method differs from most works from the literature in that we first characterize passivity from single-agent level, and then address the stability based on the output feedback interconnection model of these agents over networks. This method compares favorably to some other works \cite{gharesifard2014distributed,kia2015distributed,deng2017distributed} since it bypasses coordinate transformation and preserves network structures in convergence analysis. Besides, it also allows potential applications of mature passivity-based techniques in the study of network issues arising in distributed optimization.

A preliminary version of this work appeared in \cite{li2019input}, where only the IFP-based algorithm has been proposed. In this work, we propose the IFP-based algorithm with a possibly time-varying coupling parameter, construct more practical conditions that are easier to verify in a distributed sense, and show exponential convergence of the IFP-based algorithm over strongly connected digraphs. Moreover, a fully distributed algorithm is proposed.

The rest of this paper is organized as follows. In \Cref{Section Preliminaries and Problem Formulation}, some background knowledge of convex analysis, graph theory, and passivity is reviewed and the problem formulation is given. In \Cref{Section IFP-Based Distributed Algorithm}, an IFP-based distributed algorithm is proposed and studied over weight-balanced UJSC topologies. In \Cref{Section Distributed Derivative Feedback Algorithm}, a fully distributed algorithm over weight-balanced UJSC digraphs is proposed.
In \Cref{Section Numerical Examples}, numerical examples are presented to illustrate effects of the two algorithms. Finally, the paper is concluded in \Cref{Section Conclusion}.

\section{Preliminaries and Problem Formulation}\label{Section Preliminaries and Problem Formulation}
\subsection{Notation}\label{Notation}
Let $\mathbb{R}$ and $\mathbb{Z}$ be the sets of real and integer numbers, respectively. The Kronecker product is denoted as $\otimes$. Let $\left\lVert \cdot \right\rVert$ denote the 2-norm of a vector and also the induced 2-norm of a matrix.
The determinant of a square matrix $M$ is denoted as $\text{det}(M)$.
Given a symmetric matrix $M \in \mathbb{R}^{m\times m}$, the notation $M>0$ $(M \geq 0)$ means that $M$ is positive definite (positive semi-definite).
Denote the eigenvalues of $M$ in ascending order as $s_1(M) \leq s_2(M) \leq \ldots \leq s_m(M)$.
Let $I$ and $\mathbf{0}$ denote the identity matrix and zero matrix (or vector) of proper sizes, respectively. $\mathbf{1}_m := (1,\ldots, 1)^T \in \mathbb{R}^{m}$ denotes the vector of $m$ ones.
$\text{col}(v_1,\ldots, v_m) := (v_1^T, \ldots, v_m^T)^T$ denotes the column vector stacked with vectors $v_1, \ldots, v_m$. The notation $\text{diag} \{ \alpha_1,\ldots,\alpha_m \}$ denotes a (block) diagonal matrix with its $i$th diagonal element (block) being $\alpha_i$.
The notation $\mathcal{C}^{k}$ is used to denote a $k \in \mathbb{Z}_{\geq 1}$ times continuously differentiable function.

\subsection{Convex Analysis}
A differentiable function $f: \mathbb{R}^{m} \rightarrow \mathbb{R}$ is \textit{convex} over a convex set $\mathcal{X} \subset \mathbb{R}^{m}$ if and only if $\left(\nabla f(x)-\nabla f(y)\right)^{T}(x-y)\geq 0$, $\forall x,~y \in \mathcal{X}$, and is \textit{strictly convex} if and only if the strict inequality holds for any $x \neq y$. It is \textit{$\mu$-strongly convex} if and only if $\left(\nabla f(x)-\nabla f(y)\right)^{T}(x-y)\geq \mu \|x-y\|^{2}$, $\forall x, y \in \mathcal{X}$. 
An equivalent condition for the strong convexity is the following:
$f(y) \geq f(x) + \nabla f(x)^T (y - x) + \frac{\mu}{2} \|y - x\|^2$, $\forall x, y \in \mathcal{X}$.
An operator $\mathbf{f}: \mathbb{R}^{m} \rightarrow \mathbb{R}^{m}$ is \textit{$l$-Lipschitz continuous} over a set $\mathcal{X} \in \mathbb{R}^{m}$ if $\| \mathbf{f}(x) - \mathbf{f}(y) \| \leq l \| x - y \|$, $\forall x,y\in \mathcal{X}$.

\subsection{Graph Theory}
The information exchanging network is represented by a graph $\mathcal{G} = (\mathcal{N},\mathcal{E})$, where $\mathcal{N} = \{1,\ldots,N\}$ is the node set of all agents, $\mathcal{E}\subset \mathcal{N}\times\mathcal{N}$ is the edge set. The edge $(j,i) \in \mathcal{E}$ means that agent $i$ can obtain information from agent $j$, and $j\in \mathcal{N}_{i}$, where $\mathcal{N}_{i} = \{i \in \mathcal{N} ~|~(j,i)\in \mathcal{E}\}$ is agent $i$'s neighbor set. We assume in this work that there are no self-loops in $\mathcal{G}$, i.e., $(i,i) \notin \mathcal{E}$ and $i \notin \mathcal{N}_{i}$.
The graph $\mathcal{G}$ is said to be \textit{undirected} if $(i,j) \in \mathcal{E} \Leftrightarrow (j,i) \in \mathcal{E}$ and \textit{directed} otherwise.
A sequence of successive edges $\{(i,p), (p, q),\ldots, (v, j)\}$ is a \textit{directed path} from agent $i$ to agent $j$.
$\mathcal{G}$ is said to be \textit{strongly connected} if there exists a directed path between any two agents. The adjacency matrix is defined as $\mathcal{A} = [a_{ij}]$, where $a_{ii} = 0$; $a_{ij} > 0$ if $(j,i) \in \mathcal{E}$, and $a_{ij} = 0$, otherwise. The in-degree and out-degree of the $i$th agent are $d_{in}^{i} = \sum_{j=1}^{N} a_{ij}$ and $d_{out}^{i} = \sum_{j=1}^{N} a_{ji}$, respectively. The graph $\mathcal{G}$ is said to be \textit{weight-balanced} if $d_{in}^{i} = d_{out}^{i},~\forall i \in \mathcal{N}$.
The in-degree matrix is $W_{in} = \text{diag}\{d_{in}^{1},\ldots, d_{in}^{N} \}$. The Laplacian matrix of $\mathcal{G}$ is defined as $L = W_{in} - \mathcal{A}$. When $\mathcal{G}$ is weight-balanced, it satisfies that $\mathbf{1}_{N}^T L = \mathbf{0}$ and $L \mathbf{1}_{N} = \mathbf{0}$.
A time-varying graph $\mathcal{G}(t)$ is said to be \textit{uniformly jointly strongly connected} (UJSC) if there exists a $T > 0$ such that for any $t_k$, the union $\cup_{t \in [t_k, t_k + T]} \mathcal{G}(t)$ is strongly connected.

\subsection{Passivity}
Consider a nonlinear dynamics described by
\begin{equation}\label{general nonlinear system}
\begin{cases}
\dot{x} = F \left(x, u \right)\\
y = H \left( x,u \right)
\end{cases}
\end{equation}
where $x \in \mathcal{X} \subset \mathbb{R}^n$, $u \in \mathcal{U} \subset \mathbb{R}^m$ and $y \in \mathcal{Y} \subset \mathbb{R}^m$ are the state, input and output, respectively, and $\mathcal{X}$, $\mathcal{U}$ and $\mathcal{Y}$ are the state, input and output spaces, respectively. The functions $F: \mathcal{X} \times \mathcal{U} \rightarrow \mathbb{R}^{n}$, $H: \mathcal{X} \times \mathcal{U} \rightarrow \mathcal{Y}$ represent system and output dynamics, respectively, and are assumed to be sufficiently smooth, i.e., $\mathcal{C}^{n}$ for large enough integer $n$.

Let us give the definition of passivity and input feedforward passivity for a nonlinear system based on \cite[Definition 6.3]{khalil1996noninear}, \cite[Definition 2.12]{sepulchre2012constructive}.

\begin{definition}
System \eqref{general nonlinear system} is said to be \textit{passive} from $u$ to $y$ if there exists a continuously differentiable positive semi-definite function $V(x)$, called the \textit{storage function}, such that
\begin{equation}\label{passivity definition inequality}
\dot{V} \leq u^T y, \quad \forall (x, u) \in \mathcal{X} \times \mathcal{U}.
\end{equation}
Moreover, it is said to be \textit{input feedforward passive} (IFP) if $\dot{V} \leq u^T y - \nu u^T u$, for some $\nu \in \mathbb{R}$, denoted as IFP($\nu$).
\end{definition}
The sign of the IFP index $\nu$ denotes an excess or shortage of passivity. Specifically, when $\nu > 0$, the system is said to be \textit{input strictly passive} (ISP). When $\nu < 0$, the system is said to be \textit{input feedforward passivity-short} (IFPS).
If we define a new output as $\tilde{y} = y - \nu u$, then the IFP system becomes passive from $u$ to $\tilde{y}$.
Throughout this paper, we consider the storage function to be positive definite and radially unbounded.

\subsection{Problem Formulation}\label{Subsection Problem Formulation}
Let us formulate the problem and give some necessary assumptions in this subsection.
Consider the distributed convex optimization problem among a group of agents in the node set $\mathcal{N} = \{1,\ldots,N\}$,
\begin{equation}\label{problem}
\begin{aligned}
\min_{\mathrm{x}}\sum_{i\in \mathcal{N}} f_i (\mathrm{x})
\end{aligned}
\end{equation}
where $\mathrm{x} \in \mathbb{R}^{m}$ and each local objective function $f_i: \mathbb{R}^{m} \rightarrow \mathbb{R}$ satisfies the following assumption.
\begin{assumption}\label{Assumption strongly convex}
Each $f_i(\mathrm{x})$ is $\mathcal{C}^{2}$ and $\mu_i$-strongly convex, with its gradient $\nabla f_i(\mathrm{x})$ being $l_i$-Lipschitz continuous.
\end{assumption}
This assumption also implies that $\Vert \nabla f_i(\mathrm{x}) - \nabla f_i(\mathrm{x}') \Vert \leq l_i \Vert \mathrm{x} - \mathrm{x}' \Vert$ and $\mu_i I \leq \nabla^2 f_i(\mathrm{x}) \leq l_i I$, $\forall \mathrm{x}, \mathrm{x}' \in \mathbb{R}^{m}$. Note that \Cref{Assumption strongly convex} is widely adopted in the literature, see, e.g., \cite{kia2015distributed, qu2019exponential}. It is required in this paper to ensure IFP properties and to estimate IFP indices of agents. In addition, it is shown later that the Lipschitz requirement can be relaxed by selecting proper parameters in the algorithms.

\vspace{2mm}
Under \Cref{Assumption strongly convex}, the necessary and sufficient condition of optimality for problem \eqref{problem} is \cite[Section 5.5.3]{boyd2004convex}
\begin{equation}\label{optimal solution}
\sum_{i \in \mathcal{N}} \nabla f_i (\mathrm{x}) = \mathbf{0}
\end{equation}
Denote $x_i \in \mathbb{R}^{m}$ as agent $i$'s local estimation of the global optimal solution and let $x = \text{col}(x_1,\ldots,x_N)$, then problem \eqref{problem} is equivalent to \cite{gharesifard2014distributed}
\begin{equation}\label{problem written in distributed form}
\begin{aligned}
& \min_{x} && f(x) = \sum_{i\in \mathcal{N}} f_i (x_i)\\
&\text{subject to} && x_i = x_j, \quad \forall i,j \in \mathcal{N}
\end{aligned}
\end{equation}
where the constraints are \textit{consensus constraints} for agents to reach a common value.
Under \Cref{Assumption strongly convex} and due to \eqref{optimal solution}, the optimal solution to problem \eqref{problem written in distributed form} should satisfy
\begin{equation}\label{KKT conditions}
\sum_{i \in \mathcal{N}} \nabla f_i (x_i) = \mathbf{0}, \quad
x_i = x_j, \quad \forall i,j \in \mathcal{N}.
\end{equation}

Consider the distributed optimization over UJSC weight-balanced digraphs. To the best of our knowledge, this problem has never been addressed in the continuous-time domain.

\begin{assumption}\label{Assumption switching graphs}
The agents interact with each other through a sequence of UJSC digraphs $\{ \mathcal{G}(t) \}$, where $\mathcal{G}(t)$ is weight-balanced pointwise in time and $L(t) \neq \mathbf{0}$, $\forall t \geq 0$.
\end{assumption}
This assumption does not restrict the switching logic of $\mathcal{G}(t)$ provided it is UJSC for a finite $T$. Note that the time interval $T$ is only imposed to ensure convergence performance, and our results in this work hold as long as $\mathcal{G}(t)$ is strongly connected in a probabilistic sense \cite{zhang2018distributed}. We will propose two algorithms in the following sections. The information of $L(t)$ is required for the first algorithm, while it is not used at all for the second algorithm. Here the trivial case of $L(t) = \mathbf{0}$ is omitted.

\section{IFP-Based Distributed Algorithm}\label{Section IFP-Based Distributed Algorithm}
In this section, we propose a distributed algorithm based on input feedforward passivity and study its stability over UJSC balanced topologies.
\subsection{IFP-Based Distributed Algorithm}
We propose an IFP-based distributed algorithm as follows.
\begin{algorithm}[H]
\caption{IFP-Based Distributed Algorithm}
\label{Algorithm distributed algorithm}
\begin{algorithmic}
\STATE
\begin{itemize}
\item \textbf{Initialization}:
\begin{enumerate}
\item Choose arbitrary constants $\alpha >0$, $\beta \in \mathbb{R}$, $\gamma > 0$.
\item Choose any $x_i(0) \in \mathbb{R}^{m}$, and $\lambda_i(0) \in \mathbb{R}^{m}$ such that $\sum_{i \in \mathcal{N}} \lambda_i(0) = \mathbf{0}$.
\item Evaluate $\nu_i, \forall i \in \mathcal{N}$ by \eqref{problem to get nu}.
\end{enumerate}
\item \textbf{Design}: Design $\sigma(t) > 0$ based on chosen parameters.
\item \textbf{Dynamics for agent $i$, $i \in \mathcal{N}$}:
\begin{subequations}\label{System distributed algorithm}
\begin{align}
\dot{x}_i &= -\alpha \nabla f_i(x_i) - \lambda_i + \beta u_i \label{eq: dynamics of x_i}\\
\dot{\lambda}_i &= -\gamma u_i \label{eq: dynamics of lambda_i}\\
u_i &= \sigma(t) \sum_{j\in \mathcal{N}_{i}(t)}a_{ij}(t)( x_j - x_i)\label{input local}
\end{align}
\end{subequations}
\end{itemize}
\end{algorithmic}
\end{algorithm}

$x_i, \lambda_i, \in \mathbb{R}^{m}$ and $u_i \in \mathbb{R}^{m}$ are local variables and input for the $i$th agent, respectively;
$\alpha > 0$, $\beta \in \mathbb{R}$ and $\gamma > 0$ are constant parameters and $\sigma(t) > 0$ is the coupling gain for the diffusive couplings \eqref{input local}.
To ease the discussion on parameters, we assume that $\alpha,\beta,\gamma$ are arbitrary parameters, while $\sigma(t)$ is a finite and possibly time-varying coupling gain to be designed.
The initial condition $\sum_{i\in \mathcal{N}} \lambda_i(0) = \mathbf{0}$ is required to ensure the optimality of the equilibrium point, which will be specified in subsequent analysis. A simple initial choice can be $\lambda_i(0) = \mathbf{0}$, $\forall i \in \mathcal{N}$.

Initially, each agent in \eqref{eq: dynamics of x_i} estimates the optimal value by local gradient descent. Since $f_i$, $\forall i \in \mathcal{N}$ may not be the same, an auxiliary variable $\lambda_i$ is introduced to compensate for the difference of local gradients and ensure the existence of an equilibrium. Then, a diffusive coupling protocol \eqref{input local} is added to \eqref{eq: dynamics of x_i}, \eqref{eq: dynamics of lambda_i} in order to drive the dynamics to reach a consensus on the final optimal value.
\Cref{Algorithm distributed algorithm} is a distributed algorithm since each agent only exchanges information with neighboring agents.

Denote $x = \text{col}(x_1,\ldots,x_N)$, $\lambda = \text{col}(\lambda_1,\ldots,\lambda_N)$. Agents in \Cref{Algorithm distributed algorithm} are interconnected through diffusive couplings $u_i$, $\forall i \in \mathcal{N}$. By eliminating $u_i$, the compact form of the overall closed-loop system is written as
\begin{subequations}\label{System Compact form of distributed algorithm}
\begin{align}
\dot{x} &= -\alpha \nabla f(x) - \lambda - \sigma(t) \beta \mathbf{L}(t)  x \label{dynamics x compact}\\
\dot{\lambda} &= \sigma(t) \gamma \mathbf{L}(t) x\label{dynamics lamb compact}
\end{align}
\end{subequations}
where $\mathbf{L}(t) = L(t) \otimes I_{m}$, and $L(t)$ is the graph Laplacian of $\mathcal{G}$.

\begin{remark}\label{Remark 2}
\Cref{Algorithm distributed algorithm} in the form of \eqref{System Compact form of distributed algorithm} is a generalization of algorithms developed in \cite{kia2015distributed}. Specifically, let $\sigma(t) = 1$ and $\gamma = \alpha\beta$, then \Cref{Algorithm distributed algorithm} reduces to the distributed algorithm in \cite{kia2015distributed}. When $\alpha = \gamma = \sigma(t) = 1$, and $\beta = 0$, \Cref{Algorithm distributed algorithm} reduces to the simplified algorithm in \cite{kia2015distributed}.
Compared with \cite{kia2015distributed},
\Cref{Algorithm distributed algorithm} includes more general cases whose convergence cannot be proved by methods in \cite{kia2015distributed}, e.g., when $\sigma(t)$ is time-varying, when $\beta$ is negative, and when $\gamma$ is independent of $\alpha, \beta$.
Moreover, it is shown later that this generalized algorithm is valid over UJSC topologies in addition to directed and strongly connected switching topologies \cite{kia2015distributed}, and has an exponential convergence rate when the graph is strongly connected.
\end{remark}

\begin{lemma}\label{Lemma unique optimal point}
Under Assumptions \ref{Assumption strongly convex} and \ref{Assumption switching graphs}. If there exists an equilibrium point $(x^*,\lambda^*)$ to system \eqref{System Compact form of distributed algorithm} that satisfies $\sum_{i \in \mathcal{N}} \lambda_i^* = \mathbf{0}$, where $x^* = \text{col}(x_1^*,\ldots,x_N^*)$, $\lambda^* = \text{col}(\lambda_1^*,\ldots,\lambda_N^*)$, then $(x^*,\lambda^*)$ is also unique with $x_i^*$ being the optimal solution to problem \eqref{problem}.
\end{lemma}
\begin{proof}
The equilibrium point $(x^*,\lambda^*)$ satisfies
\begin{subequations}\label{System optimal point}
\begin{align}
\dot{x}^* &= -\alpha \nabla f(x^*) - \lambda^* \equiv \mathbf{0}\label{equilibrium x star}\\
\dot{\lambda}^* &= \sigma(t) \gamma \mathbf{L}(t) x^* \equiv \mathbf{0}\label{equilibrium lambda star}
\end{align}
\end{subequations}
where the term $- \sigma(t) \beta \mathbf{L}(t) x^*$ in \eqref{equilibrium x star} is zero and omitted since \eqref{equilibrium lambda star} implies $\mathbf{L}(t) x^* \equiv \mathbf{0}$. Since the graph is UJSC, $\mathbf{L}(t) x^* \equiv \mathbf{0}$ for all $t$ implies that $ x_i^* =  x_j^*,~\forall i, j\in \mathcal{N}$.
Next, multiplying \eqref{equilibrium x star} by $(\mathbf{1}_{N} \otimes I_m)^T$ from the left, one has,
\begin{align*}
& -(\mathbf{1}_{N} \otimes I_m)^T \alpha \nabla f(x^*) - (\mathbf{1}_{N} \otimes I_m)^T \lambda^*\\
=& -\sum_{i\in\mathcal{N}} \alpha \nabla f_i (x_i^*) - \sum_{i \in \mathcal{N}} \lambda_i^*\\
=& -\alpha \sum_{i\in\mathcal{N}} \nabla f_i (x_i^*) \equiv \mathbf{0}
\end{align*}
which satisfies \eqref{KKT conditions}. Therefore, $x_i^*$ is the optimal solution to problem \eqref{problem}.
Besides, the strong convexity of $f(x)$ in \Cref{Assumption strongly convex} implies that $x^*$ is unique \cite[Section 9.1.2]{boyd2004convex}. Thus, by \eqref{equilibrium x star}, $\lambda^*$ is unique as well.
\end{proof}
Hereafter, we call $(x^*,\lambda^*)$ the \textbf{optimal point}. The convergence of \Cref{Algorithm distributed algorithm} will be addressed in \Cref{Subsection Algorithm 1 Over UJSC Balanced Topologies}.

\subsection{Input Feedforward Passivity of the Error System}

In this subsection, we show that the error subsystem of each agent inherits the input feedforward passivity, which is a crucial step before the convergence analysis over UJSC balanced digraphs, and the design of a passivated algorithm in \Cref{Section Distributed Derivative Feedback Algorithm}.

By \Cref{Lemma unique optimal point}, for agent $i$, one has
\begin{subequations}\label{System optimal point for each agent}
\begin{align}
\dot{x}_i^* &= -\alpha \nabla f_i(x_i^*) - \lambda_i^* \equiv \mathbf{0}\\
\dot{\lambda}_i^* & \equiv \mathbf{0}.
\end{align}
\end{subequations}
Denote $\Delta x_i = x_i-x_i^*$, $\Delta \lambda_i = \lambda_i - \lambda_i^*$.
Then, the group of error subsystems between \eqref{System distributed algorithm} and \eqref{System optimal point for each agent} is
\begin{equation}\label{subsystems}
\begin{aligned}
&\Sigma_i, ~\forall i \in \mathcal{N}:\\
& \begin{cases}
\Delta \dot{x}_i & = -\alpha \left( \nabla f_i(x_i)- \nabla f_i(x_i^*)\right) - \Delta\lambda_i + \beta u_i\\
\Delta \dot{\lambda}_i & = - \gamma u_i\\
y_i & = \Delta x_i
\end{cases}
\end{aligned}
\end{equation}
where $y_i$ is defined as the output of the $i$th error subsystem. Then the input $u_i$, $\forall i \in \mathcal{N}$ can be rewritten as
\begin{equation}\label{diffusive couplings}
u_i = \sigma(t) \sum_{j\in \mathcal{N}_{i}}a_{ij}(t)(y_j - y_i),\quad \forall i \in \mathcal{N}
\end{equation}
or $u = - \sigma(t) \mathbf{L}(t)y$ in a compact form, where $u = \text{col}(u_1, \ldots, u_N)$, $y = \text{col}(y_1,\ldots, y_N)$.
Assume that, corresponding to the real agents, there exists a group of virtual agents such that the $i$th virtual agent possesses the subsystem $\Sigma_i$. Then, \Cref{Algorithm distributed algorithm} can be seen as output feedback interconnections of these virtual agents. In fact, no information of $(x_i^*, \lambda_i^*)$ is needed for communication since $y_i- y_j = \Delta x_i - \Delta x_j = x_i - x_j$. Then, each agent possesses the same information as its corresponding virtual agent.

Next, we show that each error subsystem $\Sigma_i$ in \eqref{subsystems} is IFP($\nu_i$) with index $\nu_i \leq 0$.

\begin{lemma}\label{Lemma nonlinear IFP}
Under \Cref{Assumption strongly convex}, each error subsystem $\Sigma_i$ in \eqref{subsystems} is IFP($\nu_i \leq 0$) from input $u_i$ to output $y_i$ with respect to the storage function
\[
\begin{array}{ll}
V_i = & \frac{\eta_i}{2} \|z_i\|^2 - \frac{1}{\gamma} \Delta x_i^T \Delta \lambda_i + \frac{\alpha}{\gamma}\left( f_i(x_i^*) - f_i(x_i) \right)\\
& + \frac{\alpha}{\gamma} \nabla f_i(x_i^*)^T \Delta x_i \numberthis \label{Storage function}
\end{array}
\]
where $\eta_i > \frac{1}{\mu_i \alpha \gamma}$ and $z_i = \alpha \left( \nabla f_i(x_i)- \nabla f_i(x_i^*)\right) + \Delta\lambda_i$.
\end{lemma}

\begin{proof}
See the Appendix.
\end{proof}

As pointed out by \cite{li2019consensus}, it is in general difficult to derive the exact IFP index for a nonlinear
system, and only its lower bound can be obtained by specifying
the storage function. With the storage function \eqref{Storage function}, the lower bound of the exact IFP index can be obtained locally by solving the minimax problem
\begin{equation}\label{problem to get nu}
\nu_i = -\min_{\eta_i} \max_{x_i} \frac{\left\lVert \eta_i \left(\alpha \beta \nabla^2 f_i(x_i) - \gamma I \right) - \frac{\beta}{\gamma}I \right\rVert^2 }{4\left( \mu_i \eta_i \alpha - \frac{1}{\gamma}\right)}.
\end{equation}
When each $f_i$ is quadratic, $\forall i \in \mathcal{N}$, the error system \eqref{subsystems} becomes a linear system. The exact IFP index for a linear system can be easily obtained by solving an LMI related to the positive real lemma \cite[Lemma 2]{kottenstette2014relationships}.
The problem of reducing this gap between the lower bound and the exact index of IFP remains open and is left for the future work.


\begin{remark}
It is in general not difficult to obtain $\nu_i$ by solving \eqref{problem to get nu} since local objective functions are usually of simple forms. Even when the local objective functions are complicated, problem \eqref{problem to get nu} can be relaxed to 
\begin{align*}
\nu_i \geq & -\min_{\eta_i} \max_{x_i} \frac{ \left( \left\| \eta_i \alpha \beta \nabla^2 f_i(x_i) \right\| + \left\| \eta_i \gamma I \right\| + \left\| \frac{\beta}{\gamma}I \right\| \right)^2 }{4\left( \mu_i \eta_i \alpha - \frac{1}{\gamma}\right)}\\
\geq & - \min_{\eta_i} \frac{ \left( \left( \eta_i \alpha l_i + \frac{1}{\gamma}\right) | \beta | + \eta_i \gamma \right)^2 }{4\left( \mu_i \eta_i \alpha - \frac{1}{\gamma}\right)} \numberthis \label{lower bound of nu}
\end{align*}
where $l_i$ is the Lipschitz index defined in \Cref{Assumption strongly convex}. Here \eqref{lower bound of nu} can be easily solved, providing a lower bound of the exact IFP index, which we can denote as the new $\nu_i$.
It can also be observed that when $\beta = 0$, \eqref{problem to get nu} reduces to $\nu_i = -\min_{\eta_i} \frac{\eta_i^2 \gamma^2}{4\left( \mu_i \eta_i \alpha - \frac{1}{\gamma}\right)} = - \frac{\gamma}{\mu_i^2 \alpha^2}$. The IFP index of agent $i$ is only related to the strong convexity index $\mu_i$. In this case, the Lipschitz continuity of the gradients is not required.
\end{remark}

\subsection{Algorithm Over UJSC Balanced Topologies}\label{Subsection Algorithm 1 Over UJSC Balanced Topologies}
In this subsection, we analyze the convergence of \Cref{Algorithm distributed algorithm} over weight-balanced and UJSC switching topologies based on output feedback interconnections of subsystems $\Sigma_i$ in \eqref{subsystems}. Meanwhile, the effort in constructing candidate Lyapunov functions in convergence analysis is greatly reduced.

\begin{definition}\label{Definition output consensus}
The group of agents $\Sigma_i$, $\forall i \in \mathcal{N}$ is said to achieve output consensus if their outputs satisfy $\lim_{t\rightarrow \infty} \left\lVert y_i(t) - y_j(t)\right\rVert = 0, ~\forall i,j \in \mathcal{N}$.
\end{definition}

\begin{theorem}[\hspace{1sp}\cite{li2019input}]\label{Theorem switching graphs}
Under Assumptions \ref{Assumption strongly convex} and \ref{Assumption switching graphs}, the states of \Cref{Algorithm distributed algorithm} will converge to the optimal point and solve problem \eqref{problem} if $\sum_{i\in \mathcal{N}} \lambda_i(0) = \mathbf{0}$ and the coupling gain $\sigma(t)$ satisfies 
\begin{equation}\label{coupling gain switching graphs}
0 < \sigma(t) < \frac{s_{+} \left( L(t) + L^T(t) \right)}{- 2 \bar{\nu} s_{N}\left( L^T(t) L(t) \right)},~\forall t > 0
\end{equation}
where $\bar{\nu} < 0$ is the smallest value of IFP index $\nu_i,~i \in \mathcal{N}$, $s_{+}(\cdot)$ denotes the nonzero smallest eigenvalue, and $s_{N}(\cdot)$ was defined in \Cref{Notation}.
\end{theorem}
It can be proved through the Lyapunov function $V = \sum_{i \in \mathcal{N}} V_i$, where $V_i$ was defined in \eqref{Storage function}, and by the fact that $L(t) + L^T(t)$ and $L^T(t)L(t)$ have the same null space.
The details of the proof with constant $\sigma$ can be found in the conference paper \cite{li2019input}. Condition \eqref{coupling gain switching graphs} requires the calculation of eigenvalues, which may be difficult to verify in a large-scale network. Thus, a more practical condition is derived in a different manner as follows, which is easier to verify or estimate for the design of the coupling gain in a distributed sense.
\begin{theorem}\label{Theorem switching graphs 2}
Under Assumptions \ref{Assumption strongly convex} and \ref{Assumption switching graphs}, the states of \Cref{Algorithm distributed algorithm} with initial condition $\sum_{i\in \mathcal{N}} \lambda_i(0) = \mathbf{0}$ will converge to the optimal point and solve problem \eqref{problem} if the positive coupling gain $\sigma(t)$ satisfies \begin{equation}\label{coupling gain switching graphs 2}
\frac{1}{2} - \sigma(t) |\nu_i| d^{i}(t) > 0, ~\forall i \in \mathcal{N}
\end{equation}
where $d^{i}(t)$ is the in/out-degree of the $i$th agent.
\end{theorem}
\begin{proof}
Let $V = \sum_{i \in \mathcal{N}} V_i$, where $V_i$ was defined in \eqref{Storage function}. By the proof of \Cref{Lemma nonlinear IFP}, $\left\lVert
\begin{smallmatrix}
\Delta x \\ 
\Delta \lambda
\end{smallmatrix} \right\rVert \rightarrow \infty \Rightarrow V \rightarrow \infty$, thus, $V$ is radially unbounded. Suppose \eqref{coupling gain switching graphs 2} holds. Then, following \Cref{Lemma nonlinear IFP}, the derivative of $V$ gives
\begin{align*}
\dot{V} \leq & \sum_{i \in \mathcal{N}} y_i^T u_i - \nu_i u_i^T u_i \\
= & \sigma(t) \sum_{i \in \mathcal{N}} \sum_{j\in \mathcal{N}_i(t)} a_{ij}(t) y_i^T (y_j - y_i) - \nu_i u_i^T u_i\\
= & - \frac{\sigma(t)}{2} \sum_{i \in \mathcal{N}} \sum_{j\in \mathcal{N}_i(t)} a_{ij}(t) \left( y_i^T y_i - 2 y_i^T y_j + y_j^T y_j \right)\\
& - \frac{\sigma(t)}{2} \sum_{i \in \mathcal{N}} \sum_{j\in \mathcal{N}_i(t)} a_{ij}(t) \left( y_i^T y_i - y_j^T y_j \right)- \nu_i u_i^T u_i\\
= & - \frac{\sigma(t)}{2} \sum_{i \in \mathcal{N}} \sum_{j\in \mathcal{N}_i(t)} a_{ij}(t) \left\lVert y_j - y_i \right\rVert^2 \\
& - \frac{\sigma(t)}{2} \left(\mathbf{1}^T_{N} \otimes I_{m}\right) \mathbf{L}(t) \left( Y^T Y \right)\\
& - \sum_{i \in \mathcal{N}} \nu_i \left\lVert \sigma(t) \sum_{j\in \mathcal{N}_i(t)} a_{ij}(t) \left( y_j - y_i \right) \right\rVert^2\\
= & - \frac{\sigma(t)}{2} \sum_{i \in \mathcal{N}} \sum_{j\in \mathcal{N}_i(t)} a_{ij}(t) \left\lVert y_j - y_i \right\rVert^2 \\
& - \sigma^2(t)  \sum_{i \in \mathcal{N}}\nu_i \left\lVert \sum_{j\in \mathcal{N}_i(t)} a_{ij}^{\frac{1}{2}}(t) \cdot a_{ij}^{\frac{1}{2}}(t) \left( y_j - y_i \right) \right\rVert^2\\
\leq & - \frac{\sigma(t)}{2} \sum_{i \in \mathcal{N}} \sum_{j\in \mathcal{N}_i(t)} a_{ij}(t) \left\lVert y_j - y_i \right\rVert^2 \\
& - \sigma^2(t) \sum_{i \in \mathcal{N}}\nu_i \left( \sum_{j\in \mathcal{N}_i(t)} a_{ij}(t) \right) \sum_{j\in \mathcal{N}_i(t)} a_{ij}(t) \left\lVert y_j - y_i \right\rVert^2\\
= & - \sigma(t) \sum_{i \in \mathcal{N}} \left( \frac{1}{2} - \sigma(t) |\nu_i| d^{i}(t) \right) \sum_{j\in \mathcal{N}_i(t)} a_{ij}(t) \left\lVert y_j - y_i \right\rVert^2\\
\leq & 0 \numberthis\label{Derivative of V}
\end{align*}
where $Y^T Y := col\left( y_1^T y_1, \ldots, y_{N}^T y_{N}\right)$, the third equality follows from \eqref{diffusive couplings}, the fourth equality follows from $\left(\mathbf{1}^T_{N} \otimes I_{m}\right) \mathbf{L}(t) = \left(\mathbf{1}^T_{N} L(t) \right)\otimes I_{m} = \mathbf{0}$, the second inequality follows from the Cauchy-Schwarz inequality, and the last inequality follows from \eqref{coupling gain switching graphs 2}.

Then $\lim_{t \rightarrow + \infty} V(t)$ exists and is finite. $\dot{V} \leq 0$ implies that the states $\Delta x$, $\Delta \lambda$ are bounded.
The systems trajectories are bounded within the domain $\mathcal{S}_0 = \{\left(\Delta x,\Delta \lambda \right)~|~ V(t) \leq V(0)\}$.
By the first term in $\dot{V}$ (see \eqref{Derivative of V_i} in the appendix) and the jointly connectedness of $\mathcal{G}(t)$, $\dot{V} \equiv 0$ only if $z_i =\mathbf{0}$ and $y_i = y_j$, $\forall i, j \in \mathcal{N}$, where $z_i$ was defined in \Cref{Lemma nonlinear IFP}. Define the domain $\mathcal{S}_{z} := \left\{ \left(\Delta x,\Delta \lambda \right) ~|~ z_i=\mathbf{0}, y_i= y_j, \forall i,j\in \mathcal{N} \right\}$.
Clearly, $\left\|\begin{smallmatrix} \Delta \dot{x}\\ \Delta \dot{\lambda}\end{smallmatrix} \right\|$ is bounded for any bounded $\Delta x$, $\Delta \lambda$.
Invoking the LaSalle's Invariance Principle for nonautonomous systems \cite{barkana2014defending}, we conclude that the system states ultimately reach the domain $\mathcal{S}_0 \cap \mathcal{S}_{z}$.
Then output consensus is achieved by \Cref{Definition output consensus}. Recalling \eqref{diffusive couplings}, one has $u = \mathbf{0}$ when output consensus is achieved.
Therefore, $\Delta \dot{x} \rightarrow \mathbf{0}$, $\Delta \dot{\lambda} \rightarrow \mathbf{0}$, or equivalently, $\dot{x} \rightarrow \mathbf{0}$, $\dot{\lambda} \rightarrow \mathbf{0}$ as $t \rightarrow \infty$, i.e., the states of \eqref{System Compact form of distributed algorithm} asymptotically converge to an equilibrium point.

Since $\lambda - \lambda(0) = \int_{0}^{t} \dot{\lambda}(\tau) d\tau$, given the initial condition $\sum_{i\in \mathcal{N}} \lambda_i(0) = \mathbf{0}$, 
\[ 
\begin{array}{ll}
&(\mathbf{1}_{N} \otimes I_m)^T \lambda\\
=& (\mathbf{1}_{N} \otimes I_m)^T \left( \displaystyle \int_{0}^{t} \sigma(\tau)\gamma \mathbf{L}(\tau) x(\tau) d\tau + \lambda(0)\right)\\
=& \gamma \displaystyle \int_{0}^{t}\sigma(\tau) (\mathbf{1}_{N} \otimes I_m)^T (L(\tau) \otimes I_{m}) x(\tau) d\tau\\
& + \displaystyle \sum_{i \in \mathcal{N}} \lambda_i(0)\\
= & \gamma \displaystyle \int_{0}^{t}\sigma(\tau) (\mathbf{1}_{N}^T L(\tau) \otimes I_{m}) x(\tau) d\tau
= \mathbf{0} \numberthis \label{eq: sum of lambda equals zero}
\end{array}
\]
where the third equality follows from the Kronecker product and the last follows from $\mathbf{1}_{N}^T L(\tau) = \mathbf{0}$. Then \Cref{Lemma unique optimal point} holds, implying that the equilibrium point is the unique optimal point. Consequently, the states of \Cref{Algorithm distributed algorithm} will asymptotically converge to the optimal point.
\end{proof}

The case of fixed directed topologies can be seen as a special case of switching topologies, then the convergence is also guaranteed.
Readers can refer to our conference paper \cite{li2019input} for more technical details.
In addition, an exponential convergence rate can be obtained when the graph is strongly connected, as stated in the following.
\begin{theorem}\label{Theorem exponentially convergence}
Suppose that the conditions in \Cref{Theorem switching graphs 2} hold. In addition, if the communication digraph $\mathcal{G}$ is fixed, strongly connected and weight-balanced, and the coupling gain $\sigma(t) \geq \underline{\sigma} > 0$ for a constant $\underline{\sigma}$, then the states of \Cref{Algorithm distributed algorithm} with initial condition $\sum_{i\in \mathcal{N}} \lambda_i(0) = \mathbf{0}$ will exponentially converge to the optimal point.
\end{theorem}
\begin{proof}
See the Appendix.
\end{proof}
It can be observed from the proof that the exponential convergence also holds over time-varying strongly connected weight-balanced graphs provided $ 0< \underline{d} \leq  d^i(t) \leq \bar{d}$, $\forall i \in \mathcal{N}$ for some constants $\underline{d}$, $\bar{d}$.

\begin{remark}
Note that only weight-balanced graphs are considered here. The consensus over unbalanced graphs can be guaranteed similarly \cite{qu2014modularized, li2019consensus} with $V = \sum_{i \in \mathcal{N}} \xi_i V_i$, where $\xi_i > 0$ is the $i$th element of the left eigenvalue of $L$. However, the sum of local objective functions will have a shift from the global optimum \cite{xie2018distributed}. Thus, some modification is needed. This problem may be solved by adding a state to estimate the left eigenvalues of $L$ (e.g., \cite{zhu2018continuous}), which we will leave to future work.
\end{remark}
\subsection{Discussion on the coupling gain}\label{Subsection Discussion on the coupling gain}
In this subsection, we proceed to discuss the parameters and the design of the coupling gain $\sigma(t)$ for \Cref{Algorithm distributed algorithm}.

By \Cref{Lemma nonlinear IFP}, the subsystem $\Sigma_i$ is IFP regardless of values of $\alpha,\beta,\gamma$. Let $\sigma_{e} = \frac{1}{ 2 \max_{i}\{ d^{i}(t) |\nu_i|\}}$ be the threshold of $\sigma(t)$ when $\max_{i}\{ d^{i}(t) |\nu_i|\} \neq 0$. Clearly, $\sigma_{e} > \sigma(t) > 0$ by \Cref{Theorem switching graphs 2}, meaning that there always exists a small enough $\sigma(t)$ to synchronize the outputs. Thus, $\alpha,\beta,\gamma$ can be arbitrarily chosen within the range specified in \Cref{Algorithm distributed algorithm}.
Intuitively, the larger $\alpha$, $\beta$, $\gamma$ are, the faster the convergence rate is. However, the choices of these parameters will affect the IFP index by \eqref{problem to get nu}, and hence affect the feasible range of $\sigma(t)$.

In fact, for proper parameters, there is usually a wide feasible range for the coupling gain.
Let us take for instance the quadratic functions (i.e., linear time-invariant systems in \eqref{subsystems}) from the perspective of passivity, with $\alpha = \beta = \gamma = 1$. When the strong convexity parameter $\mu_i \gg 1$, it can be shown by solving the LMI in \cite[Lemma 2]{kottenstette2014relationships} that the IFP index $\nu_i$ is infinitesimal for each agent. Therefore, $\sigma_e$ can be arbitrarily large based on the above theorems, which corresponds with the observation in \cite[Remark 2]{kia2015distributed}, where it is said that $\sigma$ can be chosen to be any positive value for the algorithm to converge in numerical examples, rendering fully distributed in practice.
However, this is in general not true.
When $\mu_i \ll 1$, each agent is IFPS with a large-magnitude index, which indicates that the coupling gain cannot be arbitrarily large. The trajectories of systems are not guaranteed to converge if $\sigma$ is not within the feasible range. A numerical example is shown in \Cref{Section Numerical Examples} for this discussion.
Consequently, the design of coupling gain is not fully distributed and requires global information like Laplacian eigenvalues or in/out-degrees.

\Cref{Theorem switching graphs,Theorem switching graphs 2} provide sufficient conditions for convergence for \Cref{Algorithm distributed algorithm}, where the former requires eigenvalues related to the Laplacian and the latter only uses in/out-degrees.
The calculation of eigenvalues could be time-consuming, especially in a large-scale network. There are many distributed algorithms to estimate Laplacian eigenvalues, e.g., \cite{franceschelli2013decentralized,charalambous2015distributed}.
However, these algorithms are obviously not as simple as obtaining the maximum value of $d^i |\nu_i|$ needed in \Cref{Theorem switching graphs 2} by locally comparing them among neighboring agents. For example, when the graph is fixed and strongly connected, let $D_i(0) = d^i | \nu_i|$ for agent $i$, $\forall i \in \mathcal{N}$, and consider the following iteration,
\begin{equation}\label{eq: finding largest value}
D_i(k+1) = \max \left\{ D_i(k), \{ D_j(k), ~ \forall j \in \mathcal{N}_i \} \right\}, \forall i \in \mathcal{N}.
\end{equation}
Obviously, the result can be obtained in a finite number of iterations since it is simply broadcasting the maximum value along the directed path. The number of iterations needed is no greater than the longest path in the graph, while the longest path is no greater than the total number of nodes. Thus, the condition in \Cref{Theorem switching graphs 2} should be easier to verify than the one in \Cref{Theorem switching graphs} in a distributed sense.

\vspace{2mm}
Note that applying a time-varying $\sigma(t)$ can be beneficial. When the graph is strongly connected, all agents should select an identical coupling gain $\sigma$ to ensure optimality. However, when the graph $\mathcal{G}(t)$ is not strongly connected at some time $t$, it is hard to communicate and obtain an identical $\sigma$ for all agents.
In this case, we can still use \eqref{eq: finding largest value} to choose different coupling gains for agents in different disjoint subgraphs without affecting convergence to the optimal point.
Suppose that the node set $\mathcal{N}$ consists of $q(t) \geq 1$ isolated subsets and let $\mathcal{N}^k(t)$ denote the $k$th subset at time $t$.
At this time, by the weight-balanced property, all the disjoint subgraphs of $\mathcal{G}(t)$ are strongly connected respectively \cite{chopra2006passivity,li2019input}.
Then each subgroup of agents at time $t$ is considered as an isolated system, and thus convergence is guaranteed by \Cref{Theorem switching graphs 2}. Following similar lines of the proof of \Cref{Theorem switching graphs 2}, we have
$
\sum_{i \in \mathcal{N}^k(t)} \lambda_i(t) = \sum_{i \in \mathcal{N}^k(t)} \lambda_i(0)
$ and hence 
\begin{align*}
\sum_{i \in \mathcal{N}} \lambda_i(t) =\sum_{k=1}^{q(t)} \sum_{i \in \mathcal{N}^k(t)}\lambda_i(0) = 
\sum_{i \in \mathcal{N}} \lambda_i(0) = \mathbf{0}.
\end{align*}
By \Cref{Lemma unique optimal point} the optimality is preserved.

The above discussion is summarized as the following corollary.
\begin{corollary}\label{Corollary different coupling gains}
Under Assumptions \ref{Assumption strongly convex} and \ref{Assumption switching graphs}, the states of \Cref{Algorithm distributed algorithm} with initial condition $\sum_{i\in \mathcal{N}} \lambda_i(0) = \mathbf{0}$ will converge to the optimal point and solve problem \eqref{problem} if the coupling gain $\sigma_i(t)$ for agent $i$ satisfies 
\begin{align}
\frac{1}{2} - \sigma_i(t) |\nu_i| d^{i}(t) > 0, ~\forall i \in \mathcal{N}
\end{align}
and $\sigma_i(t) = \sigma_j(t)$, $\forall i, j \in \mathcal{N}^k(t)$, for all $k = 1, \ldots, q(t)$.
\end{corollary}

\section{Distributed Derivative Feedback Algorithm}\label{Section Distributed Derivative Feedback Algorithm}
Note that \Cref{Algorithm distributed algorithm} still depends on the global information $\max_{i}\{ d^{i}(t) |\nu_i|\}$.
In this section, we propose a fully distributed derivative feedback algorithm based on the passivation of \Cref{Algorithm distributed algorithm}.
\subsection{Passivation And Derivative Feedback}
The derivative feedback is widely used in distributed algorithms to ensure convergence or to modify algorithms for directed graphs \cite{antipin1994feedback,deng2017distributed,zeng2018distributeda, zeng2018distributedb}.
In this subsection, we design a new distributed algorithm and reveal that the input-feedforward passivation of IFPS agents through an internal feedforward loop is a form of derivative feedback.

Let us consider again each error subsystem $\Sigma_i$ in \eqref{subsystems}.
Suppose $\Sigma_i$ is IFP($\nu_i$), $\forall i \in \mathcal{N}$, then we apply a passivation through feedforward of input. Define a new output as $\tilde{y}_i$ for the $i$th subsystem
\begin{equation}\label{new output after passivation}
\tilde{y}_i = y_i - \nu_i u_i, \quad \forall i \in \mathcal{N}
\end{equation}
where $\nu_i \leq 0$ is the IFP index of agent $i$. The transformation is shown in \Cref{Input-feedforward-diagram}.
Obviously, the transformed system $\tilde{\Sigma}_i$ is passive.

\begin{figure}[htbp]
\centering
\includegraphics[width=0.80\linewidth,clip,keepaspectratio]{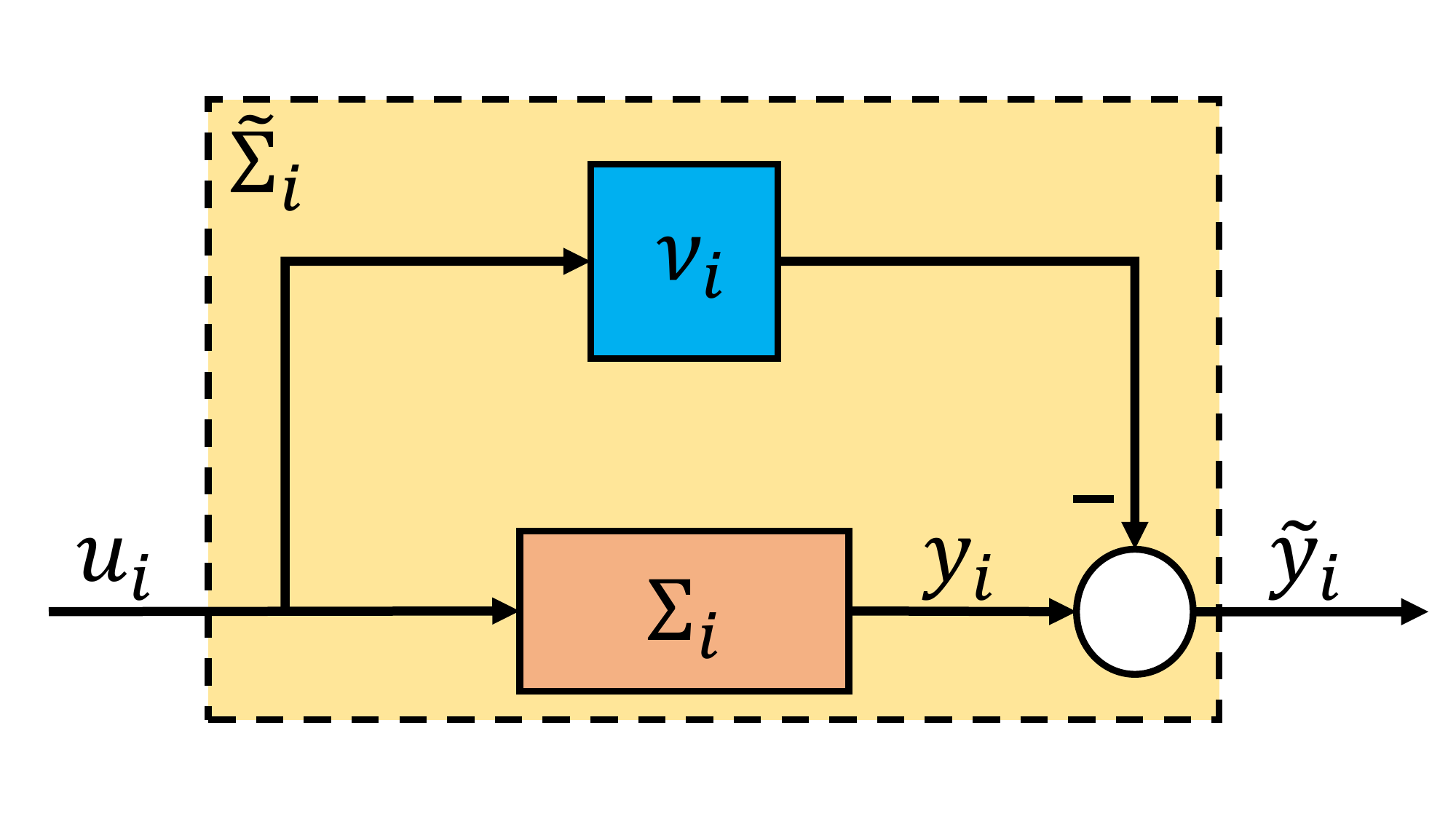}
\caption{Block diagram of the input-feedforward passivation of the $i$th virtual agent $\Sigma_i$ in \eqref{subsystems}. The notation $\tilde{\Sigma}_i$ denotes the transformed system after the input-feedforward passivation.}\label{Input-feedforward-diagram}
\end{figure}
\begin{lemma}\label{Lemma derivative feedback being pasive}
Under \Cref{Assumption strongly convex}, each subsystem $\tilde{\Sigma}_i$ defined by \eqref{subsystems} and \eqref{new output after passivation}, is passive from input $u_i$ to output $\tilde{y}_i$ with respect to the storage function \eqref{Storage function}.
\end{lemma}
\begin{proof}
Adopt the same storage function \eqref{Storage function}, then following similar lines of the proof of \Cref{Lemma nonlinear IFP}, one has $V_i \geq 0$ and
\begin{equation}\label{dot of V_i passive}
\begin{aligned}
\dot{V}_i \leq  y_i^T u_i - \nu_i u_i^T u_i
\leq  \tilde{y}_i^T u_i.
\end{aligned}
\end{equation}
\end{proof}
Adopt the diffusive couplings of $\tilde{y}_i$, $i \in \mathcal{N}$ as new inputs,
\begin{equation}\label{new input after passivation}
u_i = \sigma(t) \sum_{j\in \mathcal{N}_{i}}a_{ij}(t)\left(\tilde{y}_j - \tilde{y}_i \right),\quad \forall i \in \mathcal{N}
\end{equation}
then a novel distributed algorithm is constructed as follows.
\begin{algorithm}[H]
\caption{Distributed Derivative Feedback Algorithm}
\label{Algorithm Derivative feedback}
\begin{algorithmic}
\STATE
\begin{itemize}
\item \textbf{Initialization}:
\begin{enumerate}
\item Choose arbitrary constants $\alpha >0$, $\beta \in \mathbb{R}$, $\gamma > 0$.
\item Choose any $x_i(0) \in \mathbb{R}^{m}$, and $\lambda_i(0) \in \mathbb{R}^{m}$ such that $\sum_{i \in \mathcal{N}} \lambda_i(0) = \mathbf{0}$.
\item Each agent evaluates $\nu_i$ by \eqref{problem to get nu}.
\end{enumerate}
\item \textbf{Dynamics for agent $i$, $i \in \mathcal{N}$}:
\begin{subequations}\label{System Derivative feedback algorithm}
\begin{align}
\dot{x}_i &= -\alpha \nabla f_i(x_i) - \lambda_i + \beta u_i\\
\dot{\lambda}_i &= -\gamma u_i \label{algorithm 2 dot_lam}\\
\tilde{y}_i &= x_i - \nu_i u_i \label{algorithm 2 tilde_y}\\
u_i &= \sigma(t) \sum_{j\in \mathcal{N}_{i}(t)}a_{ij}(t)(\tilde{y}_j - \tilde{y}_i)\label{input derivative feedback}
\end{align}
\end{subequations}
\end{itemize}
\end{algorithmic}
\end{algorithm}

By eliminating $\tilde{y_i}$ and $u_i$ with \eqref{algorithm 2 tilde_y} and \eqref{algorithm 2 dot_lam}, respectively, \Cref{Algorithm Derivative feedback} can be rewritten in a compact form
\begin{subequations}\label{System Compact form of Derivative feedback algorithm}
\begin{align}
\dot{x} =& -\alpha \nabla f(x) - \lambda - \frac{\beta}{\gamma} \dot{\lambda}\\
\dot{\lambda} =& \sigma(t) \gamma \mathbf{L}(t) x + \sigma(t)\mathbf{L}(t)\nu \dot{\lambda} \label{System Compact form of dot lambda}
\end{align}
\end{subequations}
where $\nu = \text{diag}\{ \nu_1, \ldots, \nu_N \} \otimes I_{m}$. We can observe that there exist derivative feedback terms in \eqref{System Compact form of Derivative feedback algorithm}.
Since each agent only requires information from neighboring agents, \Cref{Algorithm Derivative feedback} is a distributed algorithm.

\vspace{1em}
Before proceeding to next step, note that the diffusive couplings of the new outputs bring \textit{algebraic loops} \cite[Section 8.3]{astrom2010feedback} into the overall closed-loop system. 
Thus, we have to check whether the feedback interconnection is well-posed.

The equation \eqref{System Compact form of dot lambda} can be rewritten as 
\begin{equation}
\left(I - \sigma(t) \mathbf{L}(t)\nu\right)\dot{\lambda} = \sigma(t) \gamma \mathbf{L}(t)  x
\end{equation}
Notice that $\left(I - \sigma(t) \mathbf{L}(t)\nu\right)$ should be nonsingular such that system \eqref{System Compact form of Derivative feedback algorithm} can be rewritten in the following explicit form, ensuring the well-posedness of the feedback interconnection \cite{simpson2018equilibrium}.
\begin{subequations}\label{System Explicit Compact form of Derivative feedback algorithm}
\begin{align}
\dot{x} =& -\alpha \nabla f(x) - \lambda - \sigma \beta \left(I - \sigma \mathbf{L}(t) \nu \right)^{-1} \mathbf{L}(t) x\\
\dot{\lambda} =& \sigma \gamma \left(I - \sigma \mathbf{L}(t) \nu \right)^{-1} \mathbf{L}(t) x.\label{explicit expression of dot_lambda}
\end{align}
\end{subequations}
When the IFP indices are the same, e.g., $\nu_i = \bar{\nu}$, $i \in \mathcal{N}$, where $\bar{\nu}$ was defined in \Cref{Theorem switching graphs}, then the nonsingularity of $\left(I - \sigma(t) \bar{\nu} \mathbf{L}(t) \right)$ is obvious following a matrix decomposition. However, when the IFP indices take different values, more analysis of this term is needed. We propose the following lemma.

\begin{lemma}\label{Lemma The matrix is nonsingular}
The matrix $\left(I - \sigma(t) \mathbf{L}(t) \nu \right)$ is nonsingular.
\end{lemma}
\begin{proof}
See the appendix.
\end{proof}

\subsection{Algorithm Over UJSC Balanced Topologies}
Next, we derive the following theorem stating that \Cref{Algorithm Derivative feedback} is fully distributed without global coordination.

\begin{theorem}\label{Theorem Algorithm Derivative feedback UJSC}
Under Assumptions \ref{Assumption strongly convex} and \ref{Assumption switching graphs}, the states of \Cref{Algorithm Derivative feedback} with initial condition $\sum_{i\in \mathcal{N}} \lambda_i(0) = \mathbf{0}$ will converge to the optimal point and solve problem \eqref{problem} given any coupling gain $\sigma(t) > 0$.
\end{theorem}
\begin{proof}
When $\dot{\lambda} = \mathbf{0}$, system \eqref{System Compact form of Derivative feedback algorithm} reduces to system \eqref{System Compact form of distributed algorithm}, meaning that the derivative term does not affect the equilibrium set of system \eqref{System distributed algorithm}. Besides, given the initial condition $\sum_{i\in \mathcal{N}} \lambda_i(0) = \mathbf{0}$,
\[ 
\begin{array}{ll}
&(\mathbf{1}_{N} \otimes I_m)^T \lambda\\
=& (\mathbf{1}_{N} \otimes I_m)^T 
 \left( \displaystyle \int_{0}^{t} \left( \sigma(\tau)\gamma \mathbf{L}x + \sigma(\tau) \mathbf{L} \nu \dot{\lambda} \right)d\tau + \lambda(0) \right)\\
=& \gamma \displaystyle \int_{0}^{t} \sigma(\tau) (\mathbf{1}_{N} \otimes I_m)^T (L(\tau) \otimes I_{m}) x(\tau) d\tau\\
& + \displaystyle \int_{0}^{t} \sigma(\tau) (\mathbf{1}_{N} \otimes I_m)^T (L(\tau) \otimes I_{m})\nu \dot{\lambda}(\tau) d\tau + \sum_{i \in \mathcal{N}}  \lambda_i(0)\\
= & \displaystyle \int_{0}^{t} \sigma(\tau) (\mathbf{1}_{N}^T L(\tau) \otimes I_{m}) \left( \gamma x(\tau) + \nu \dot{\lambda}(\tau) \right) d\tau\\
=& \mathbf{0}
\end{array}
\]
where the third equality follows from rules of the Kronecker product and the initial condition, the last follows from $\mathbf{1}_{N}^{T}L(\tau) = \mathbf{0}$.
It can also be shown by using the explicit expression \eqref{explicit expression of dot_lambda} of $\dot{\lambda}$ that $(\mathbf{1}_{N} \otimes I_m)^T \lambda = \mathbf{0}$, satisfying \Cref{Lemma unique optimal point}.
Thus, the equilibrium point of \Cref{Algorithm Derivative feedback} with initial condition $\sum_{i\in \mathcal{N}} \lambda_i(0) = \mathbf{0}$ is still the optimal point to the distributed optimization problem \eqref{problem}.

The information of $(x_i^*,\lambda_i^*)$ is not required for exchange. Then \Cref{Algorithm Derivative feedback} can be implemented by output feedback interconnections of virtual agents $\tilde{\Sigma}_i,~\forall i\in \mathcal{N}$.
Since $\tilde{\Sigma}_i$ is passive from input $u_i$ to output $\tilde{y}_i$ by \Cref{Lemma derivative feedback being pasive}, the consensus analysis among passive agents is similar to that among IFP agents with IFP indices being zero. Specifically, let $V = \sum_{i \in \mathcal{N}} V_i$, where $V_i$ was defined in \eqref{Storage function}. Substituting \eqref{new input after passivation} into \eqref{dot of V_i passive}, we obtain
\begin{align*}
\dot{V} \leq & \sum_{i\in \mathcal{N}} \tilde{y}_i^T u_i= \tilde{y}^T u = - \sigma(t) \tilde{y}^T \mathbf{L}(t) \tilde{y}
\leq 0
\end{align*}
where $\tilde{y} = \text{col}(\tilde{y}_1,\ldots,\tilde{y}_N)$ and $u = \text{col}(u_1,\ldots,u_N) = - \sigma(t)\mathbf{L}(t) \tilde{y}$ and the last inequality follows from the fact that $\mathbf{L}(t) \geq 0$.
Following similar lines of the proof of \Cref{Theorem switching graphs 2}, the states of \Cref{Algorithm Derivative feedback} with initial condition $\sum_{i\in \mathcal{N}} \lambda_i(0) = \mathbf{0}$ will asymptotically converge to the optimal point.
\end{proof}
Similarly, \Cref{Theorem Algorithm Derivative feedback UJSC} can be directly applied to fixed weight-balanced strongly connected digraphs as a special case of UJSC topologies, as stated in the following corollary.
\begin{corollary}\label{Corollary directed graphs for Algorithm 2}
Suppose the communication digraph $\mathcal{G}$ is fixed, strongly connected and weight-balanced. Then, under \Cref{Assumption strongly convex}, the states of \Cref{Algorithm Derivative feedback} with initial condition $\sum_{i\in \mathcal{N}} \lambda_i(0) = \mathbf{0}$ will converge to the optimal point for any $\sigma(t) > 0$.
\end{corollary}

\subsection{Discussion on the Derivative Feedback Algorithm}
Though \Cref{Algorithm Derivative feedback} requires agents to exchange with each other more information like derivatives of states, its advantages are significant.

Compared with most works on directed topologies, \Cref{Algorithm Derivative feedback} is robust against randomly changing weight-balanced digraphs with any positive coupling gain and independent of any global information. Since the time interval $T$ for UJSC graphs is not used in the proofs, $\mathcal{G}(t)$ can be relaxed to being strongly connected in a probabilistic sense, namely, it is applicable over gossip-like balanced digraphs \cite{zhang2018distributed}.
It can be observed from \Cref{Input-feedforward-diagram} that this modified algorithm can be easily realized by adding a local input feedforward loop to each subsystem $\Sigma_i$. Since the input $u_i$ of the $i$th virtual agent is the same as the input of the real agent $i$, the input feedforward of virtual agents is actually the same as the input feedforward of real agents.
Also, note that the passivation is achieved by each agent locally, no global information is needed beforehand. Thus, \Cref{Algorithm Derivative feedback} is a fully distributed algorithm.

It might be difficult to characterize the exact convergence rate of \Cref{Algorithm Derivative feedback} due to the existence of the derivative terms. Nevertheless, we will show in a numerical example in \Cref{Section Numerical Examples} that its empirical convergence rate is similar to the one of \Cref{Algorithm distributed algorithm}.


\section{Numerical Examples}\label{Section Numerical Examples}
\subsection*{Example 1}
We present a numerical example to show effects of the proposed algorithms over directed and switching topologies. Consider a network of $4$ agents possessing the following local objective functions $f_i: \mathbb{R} \rightarrow \mathbb{R}, ~i = 1,2,3,4$, respectively.
\[
\begin{array}{lll}
& f_1(\mathrm{x}) = 0.4 \mathrm{x}^2 - \mathrm{x},~ f_2(\mathrm{x}) = \ln(e^{-0.3\mathrm{x}}+ e^{0.5\mathrm{x}}) + 0.6\mathrm{x}^2\\
& f_3(\mathrm{x}) = \mathrm{x}^2 + \cos \mathrm{x},~
f_4(\mathrm{x}) = \frac{\mathrm{x}^2}{\sqrt{\mathrm{x}^2 + 1}} + 0.9 \mathrm{x}^2.
\end{array}
\]
By calculation, we obtain $\mu_1 = l_1 = 0.8$; $\mu_2 = 1.20$, $l_2 = 1.36$; $\mu_3 = 1$, $l_3 = 3$; $\mu_4 = 1.76$, $l_4 = 3.8$.
Let $\alpha = \beta = \gamma = 1$. Then, we obtain that each subsystem in \eqref{subsystems} is IFPS with $\nu_1 =-0.31$, $\nu_2 = -0.49$, $\nu_3 = -1$, and $\nu_4 = -0.68$. Next, we consider two cases of topologies.

\textbf{Case 1}: the agents are connected through a ring graph that is strongly connected and weight-balanced, as shown in \Cref{Fig_Ex1_Case_1}.
\begin{figure}[htbp]
\centering
\includegraphics[width=0.50\linewidth,clip,keepaspectratio]{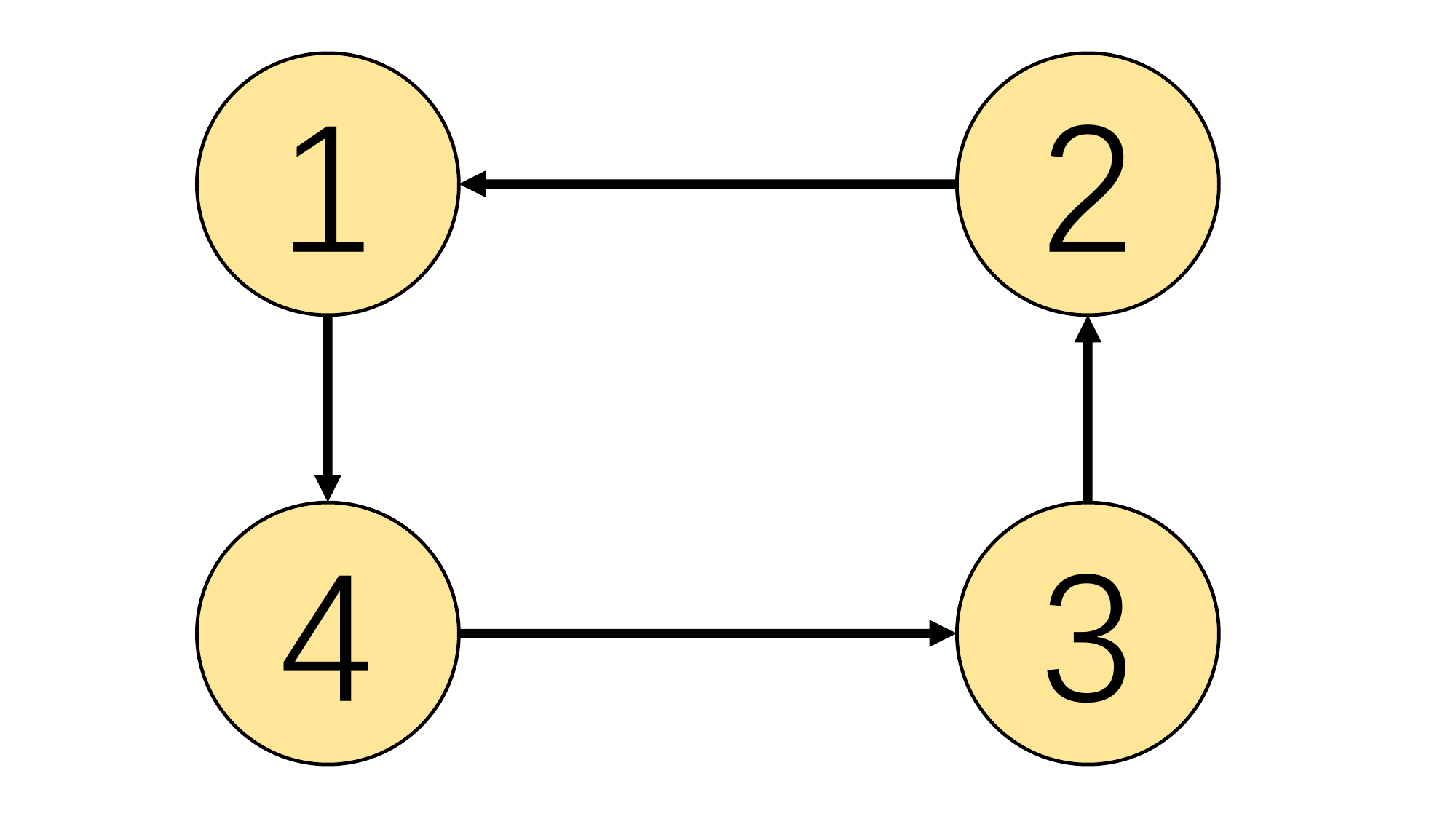}
\caption{The communication graph for the $4$ agents.}\label{Fig_Ex1_Case_1}
\end{figure}


\textbf{Case 2}: for every $0.1$ second, the graph $\mathcal{G}(t)$ switches randomly among three modes as shown in \Cref{Fig_Ex1_Case_2}.
\begin{figure}[htbp]
\centering
\includegraphics[width=1\linewidth,clip,keepaspectratio]{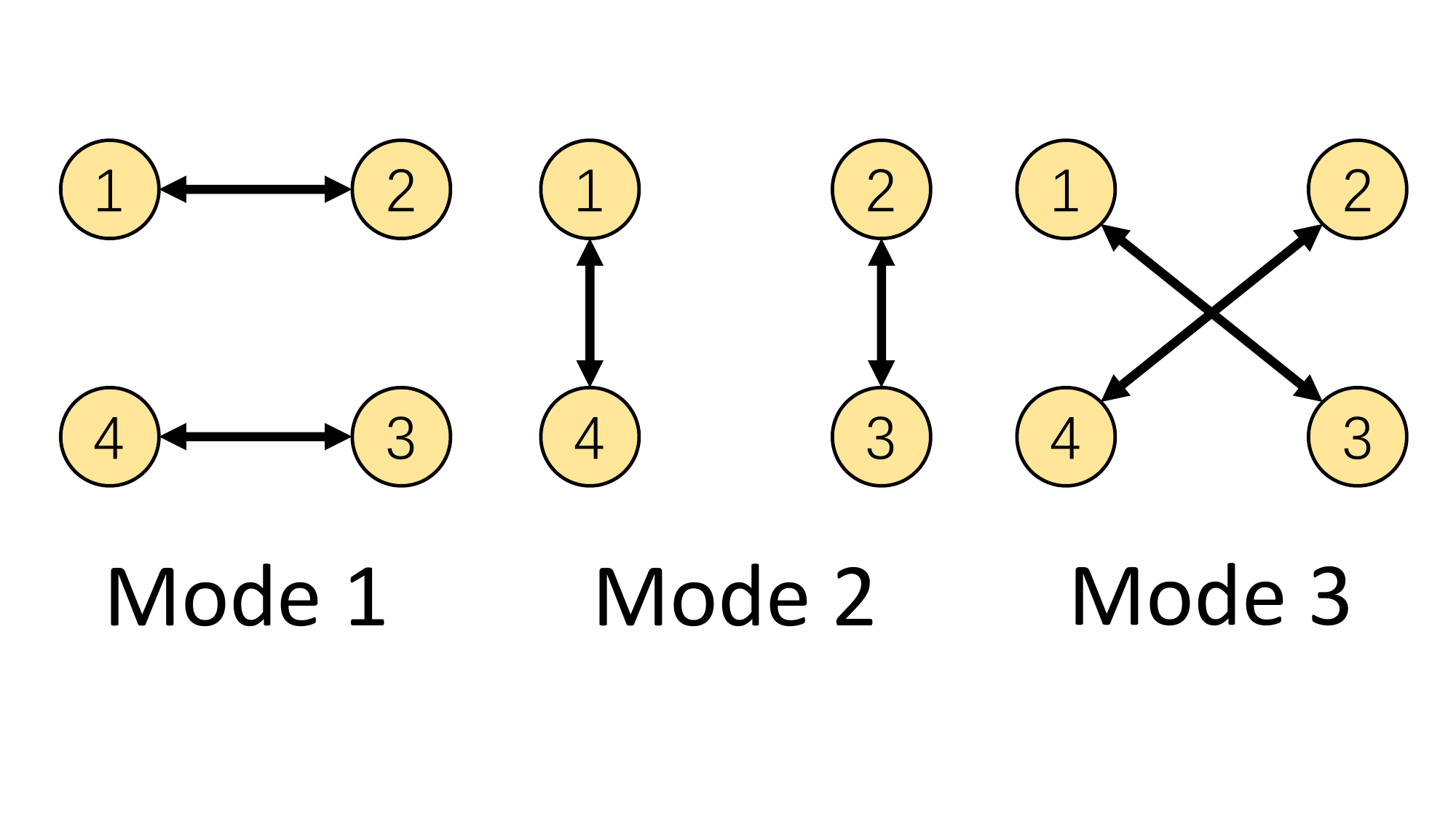}
\caption{The switching communication graph for the $4$ agents.}\label{Fig_Ex1_Case_2}
\end{figure}


The threshold coupling gains are obtained as $\sigma_{e} = 0.50$ in \eqref{coupling gain switching graphs 2} for both cases.
We implement \Cref{Algorithm distributed algorithm,Algorithm Derivative feedback} using the solver ode45 with auto-adjust variable stepsize in MATLAB over these two cases, and with $x_i(0) \in [0,1]$, $\lambda(0) = \mathbf{0}$ satisfying the initial condition.
Let $\sigma(t) = 0.35+0.1\cos(t) < \sigma_{e}$ in Case 1. To illustrate \Cref{Corollary different coupling gains}, we adopt different coupling gains for different disjoint subgraphs in Case 2. Specifically, let $\sigma_1(t) = \sigma_2(t) = 0.3 + \sin(t)$ and $\sigma_3(t) = \sigma_4(t) = 0.35 + \cos(t)$ for Mode 1; $\sigma_2(t) = \sigma_3(t) = 0.3 + \sin(t)$ and $\sigma_1(t) = \sigma_4(t) = 0.35 + \cos(t)$ for Mode 2; $\sigma_2(t) = \sigma_4(t) = 0.3 + \sin(t)$ and $\sigma_1(t) = \sigma_3(t) = 0.35 + \cos(t)$ for Mode 3.

The state trajectories are shown in \Cref{Ex1_Case1_Algorithms,Ex1_Case2_Algorithms,semilog}. It can be observed that the trajectories of $x_i$, $i = 1,2,3,4$ asymptotically converge to the optimal solution $x_i^* = 0.1601,~i = 1, 2, 3, 4$, in both cases.
The residuals $\sum_{i = 1}^{4} \| x_i - x_i^* \|$ of both algorithms over the time-invariant graph in Case 1 are shown in \Cref{semilog}. We can observe that \Cref{Algorithm Derivative feedback} has a convergence rate very similar to \Cref{Algorithm distributed algorithm} despite the existence of derivative terms, and the residuals of both algorithms decrease exponentially.

\begin{figure}[htp]
\centering
\includegraphics[width = 1\linewidth]{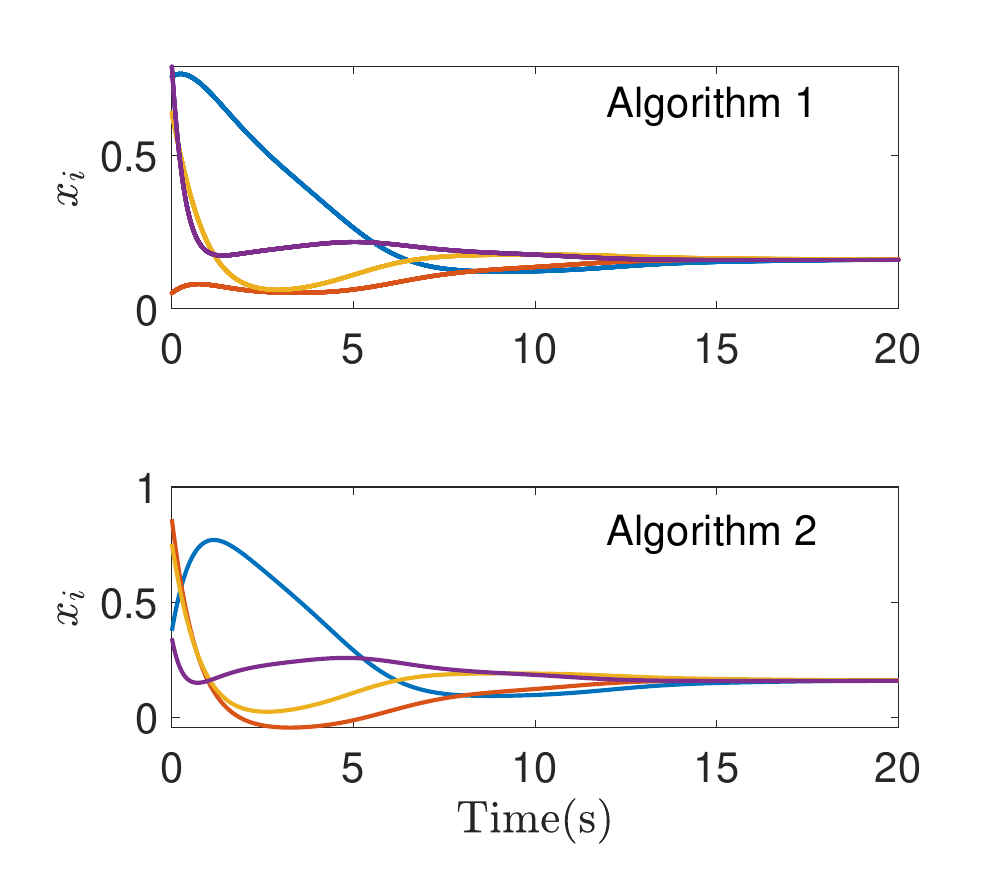}
\caption{The trajectories of $x_i$, $\forall i \in \mathcal{N}$ for the two distributed algorithms with a time-varying coupling gain over a weight-balanced digraph.}
\label{Ex1_Case1_Algorithms}
\end{figure}

\begin{figure}[htp]
\centering
\includegraphics[width = 1\linewidth]{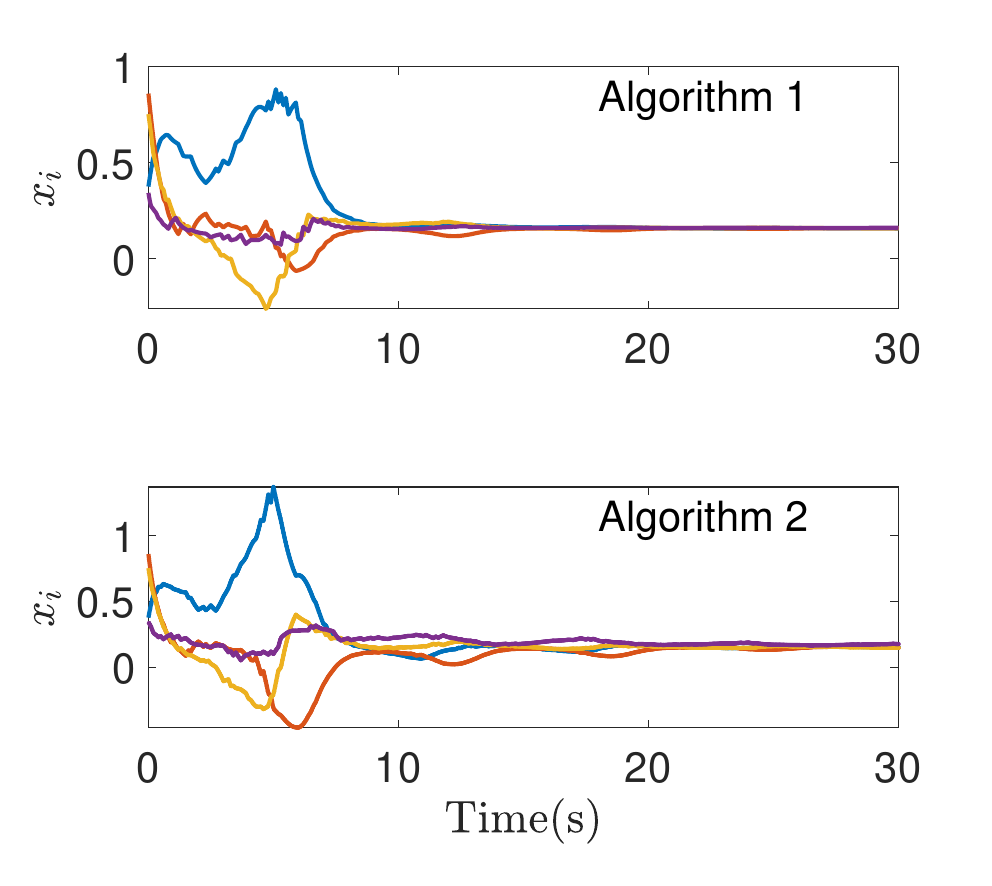}
\caption{The trajectories of $x_i$, $\forall i \in \mathcal{N}$ for the two distributed algorithms over a UJSC graph with different time-varying couplings for disjoint subgraphs.}
\label{Ex1_Case2_Algorithms}
\end{figure}

\begin{figure}[htp]
\centering
\includegraphics[width = 1\linewidth]{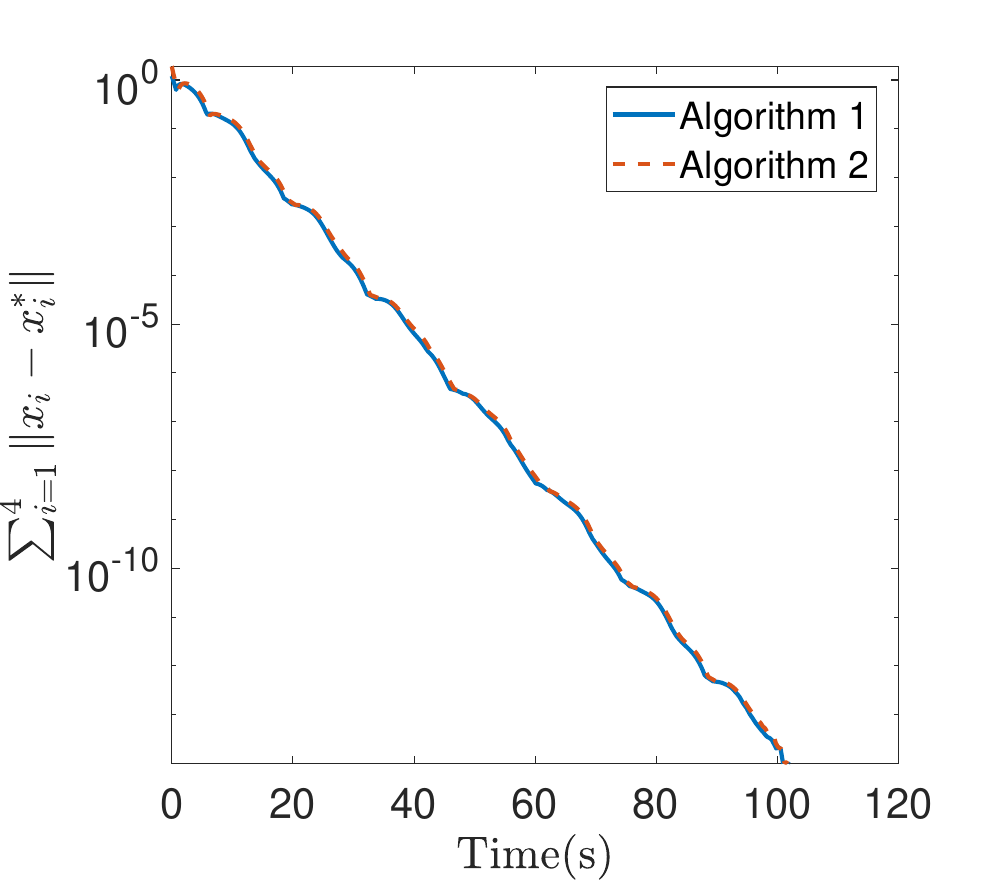}
\caption{The residuals $\sum_{i = 1}^{4} \| x_i - x_i^* \|$ for the two algorithms over the time-invariant graph in Case 1.}
\label{semilog}
\end{figure}

\subsection*{Example 2}
We present another example to compare the two algorithms. Consider a network of $4$ agents interconnected through the same graph as \Cref{Fig_Ex1_Case_1}.
The local objective functions are
$$
f_i(\mathrm{x}) = 0.025(i+1)(\mathrm{x} - i)^2, ~\mathrm{x} \in \mathbb{R},~i = 1, 2, 3, 4.
$$

Let $\alpha = \beta = \gamma = 1$. By solving the LMI in \cite[Lemma 2]{kottenstette2014relationships} with the YALMIP Toolbox \cite{lofberg2004yalmip}, we obtain that the agents are IFPS with $\nu_1 = -89.96$, $\nu_2 = -37.77$, $\nu_3 = -20.00$ and $\nu_4 = -12.00$. Then by \eqref{coupling gain switching graphs 2}, the coupling gain threshold is obtained as $\sigma_{e} = 0.0056$. According to \ref{Theorem switching graphs 2}, when $\sigma < \sigma_{e}$, the trajectories of \Cref{Algorithm distributed algorithm} will converge to the optimal point.
Then, we implement the two distributed algorithms using the solver ode45 with auto-adjust variable stepsize in MATLAB, and with $x_i(0) \in [2,3]$, $\lambda(0) = \mathbf{0}$ satisfying the initial condition. The trajectories of the two algorithms asymptotically converge to the optimal solution $x_i^* = 2.857, ~i = 1, 2, 3, 4$ when $\sigma = 0.005 \in (0,\sigma_e)$, as shown in \Cref{Ex2_Case1_Algorithms}.

When the coupling gain is outside the feasibility range, \Cref{Algorithm distributed algorithm} is not guaranteed to converge.
The error system \eqref{subsystems} is a linear system:
\begin{equation*}
\begin{bmatrix}
\Delta\dot{x}\\
\Delta\dot{\lambda}
\end{bmatrix}
=
\begin{bmatrix}
-F - \sigma L & -I\\
\sigma L & \mathbf{0}
\end{bmatrix}
\begin{bmatrix}
\Delta x\\
\Delta\lambda
\end{bmatrix}
\end{equation*}
where $F = \text{diag}\{0.1, 0.15, 0.2, 0.25\}$.
Clearly, it is unstable when $\sigma \in [0.1, 0.14]$, which accords with our discussion in \Cref{Subsection Discussion on the coupling gain}.
On the other hand, \Cref{Algorithm Derivative feedback} should be valid with any positive $\sigma$ by \Cref{Theorem Algorithm Derivative feedback UJSC}. To show this, we compare the two distributed algorithms with the same settings except $\sigma = 0.1 \notin (0,\sigma_e)$. It can be observed from \Cref{Ex2_Case2_Algorithms} that \Cref{Algorithm distributed algorithm} is unstable while the trajectories of $x_i$, $\forall i \in \mathcal{N}$ in \Cref{Algorithm Derivative feedback} asymptotically converge to the optimal solution. Here the IFP indices are different and agents are passivated locally in \Cref{Algorithm Derivative feedback}.
\begin{figure}[htp]
\centering
\includegraphics[width = 1\linewidth]{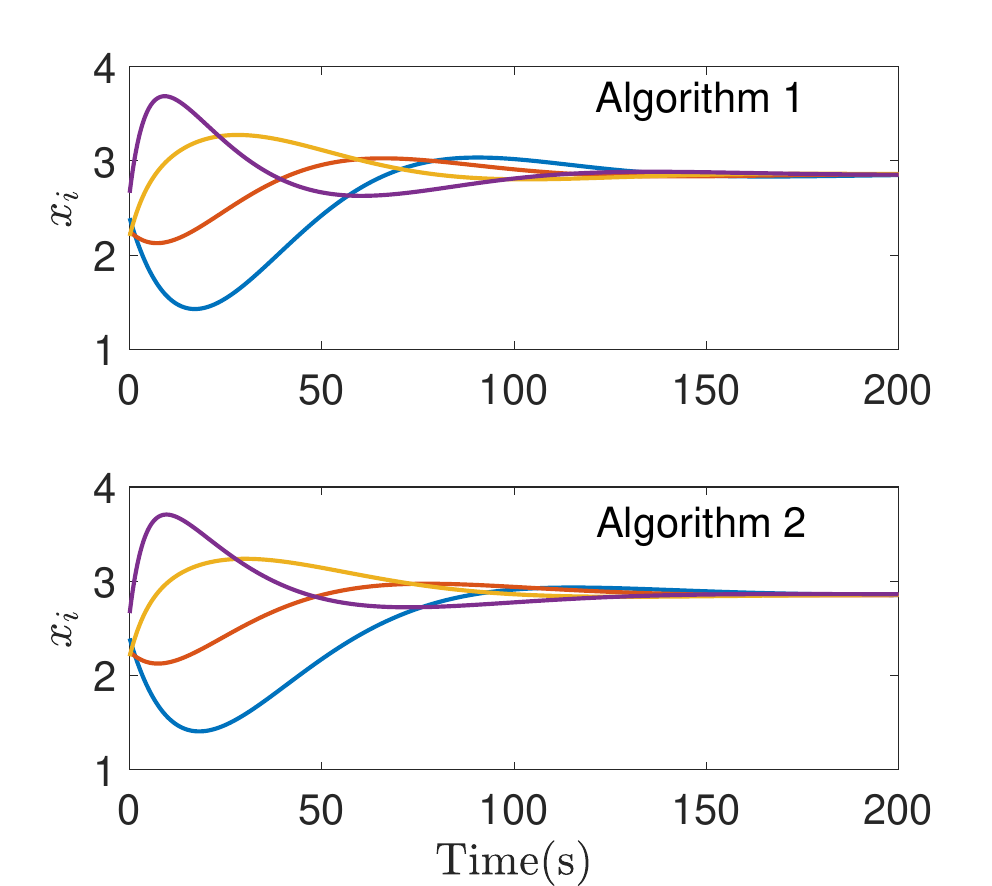}
\caption{The trajectories of $x_i$, $\forall i \in \mathcal{N}$ with $\sigma = 0.005$ for the two distributed algorithms.}
\label{Ex2_Case1_Algorithms}
\end{figure}

\begin{figure}
\centering
\includegraphics[width = 1\linewidth]{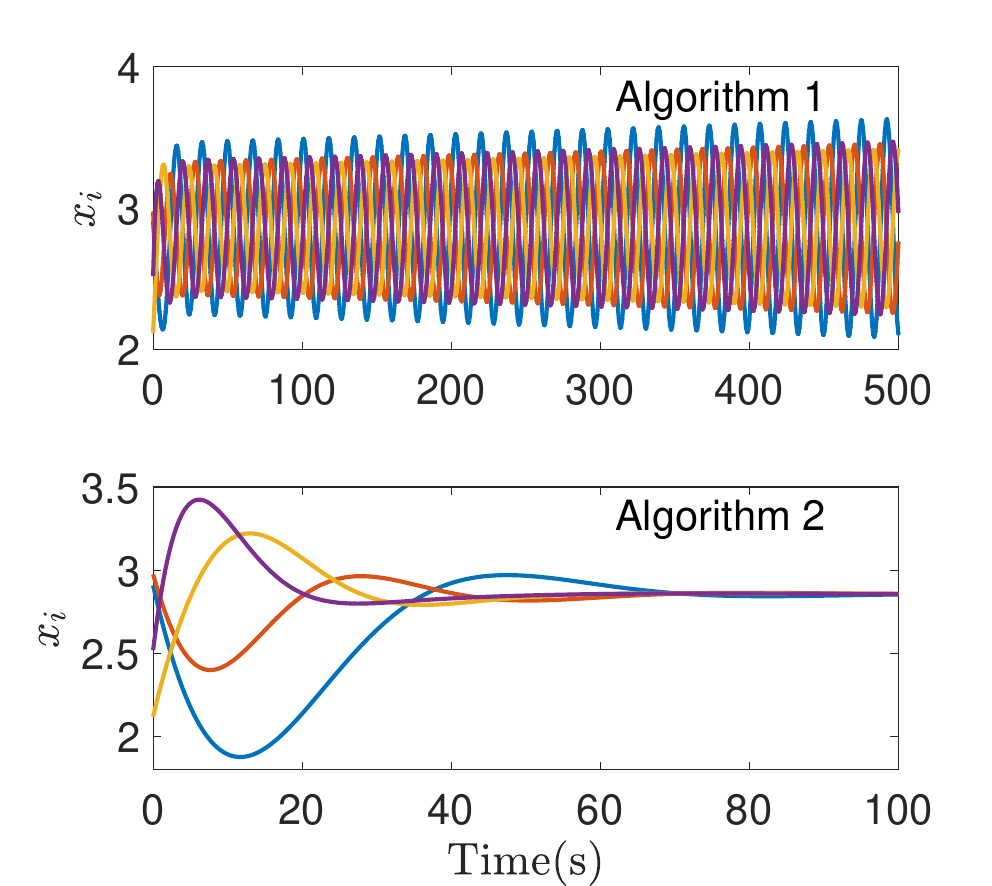}
\caption{The trajectories of $x_i$, $\forall i \in \mathcal{N}$ with $\sigma = 0.1$ for the two distributed algorithms.}
\label{Ex2_Case2_Algorithms}
\end{figure}

\subsection*{Example 3}
We present an example of $N = 100$ agents to show the scalability of the proposed algorithms. The local objective functions are
$f_i(\mathrm{x}) = b_i \mathrm{x}^2 + c_i \mathrm{x}$, $\forall i \in \mathcal{N}$, where $b_i \in [0.5,0.6]$, $c_i \in [-1, 0]$ are randomly generated with uniform distributions. The weight-balanced digraph $\mathcal{G}(t)$ is also randomly generated per second, where the probability that the edge $(i,j) \in \mathcal{E}$ exists is $0.005$, $\forall i \neq j, ~i, j \in \mathcal{N}$, and the in/out-degree $d^i$ is no greater than $\bar{d} = 2.5$, $\forall i \in \mathcal{N}$.
Let $\alpha = \gamma = \beta = 1$, then by calculation, the IFP indices satisfy $\nu_i \geq -2$, $\forall i \in \mathcal{N}$.
Thus, we can roughly obtain $\sigma \leq 0.1$ for \Cref{Algorithm distributed algorithm}.
For Algorithm 2, there is no restriction on the coupling gain, which we set as $\sigma = 1$.
We implement the two algorithms using ode45 in MATLAB with $x_i(0) \in [0,1]$, $\lambda(0) = \mathbf{0}$, and obtain \Cref{fig:algorithm}. It can be observed that the proposed algorithms have good scalability.

Note that an important feature of the algorithms is that the graph Laplacian $L(t)$ at any time can be very sparse. The time interval $T$ in \Cref{Assumption switching graphs} is imposed only to ensure the convergence performance and was not used in the proofs. In practice, our results hold as long as the graph $\mathcal{G}(t)$ is strongly connected in a probabilistic sense.
In this example, $L(t)$ usually has zero eigenvalue with a multiplicity greater than $20$, which greatly reduces communication.
\begin{figure}[htp]
\centering
\includegraphics[width = 1\linewidth]{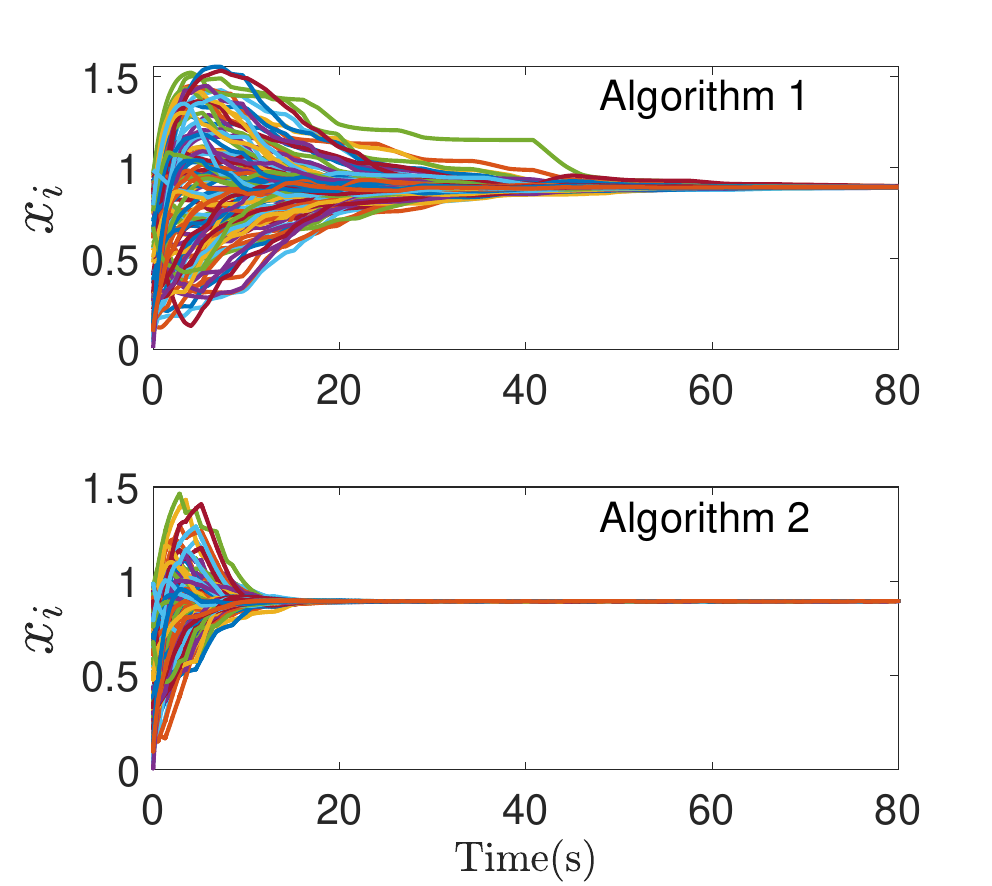}
\caption{The trajectories of $x_i$, $\forall i \in \mathcal{N}$ with $\sigma = 0.1$ for \Cref{Algorithm distributed algorithm} and $\sigma = 1$ for \Cref{Algorithm Derivative feedback}.}
    \label{fig:algorithm}
\end{figure}

\section{Conclusion}\label{Section Conclusion}
This paper has investigated a distributed optimization problem via input feedforward passivity. An input-feedforward-passivity framework has been adopted to construct a distributed algorithm that is applicable over weight-balanced digraphs. Moreover, a novel distributed derivative feedback algorithm, which is fully distributed, has been proposed via the input-feedforward passivation.
The proposed algorithms have been studied over directed and uniformly jointly strongly connected balanced topologies. Convergence conditions of a suitable coupling gain for the IFP-based distributed algorithm have been derived, while it has been shown that the distributed derivative feedback algorithm is robust against randomly changing weight-balanced digraphs with any positive coupling gain and without knowing any global information.

It is worth mentioning that there are also some limitations in this work and several directions can be considered in future work.
For instance, requiring continuous communication is difficult in practice, which can be resolved by applying discrete-time communication or discretization \cite{li2021distributed}.
Also, one could extend the IFP-based distributed algorithms to solve constrained problems or enhance robustness. Lastly, one could consider relaxing the strong convexity requirement using more advanced feedback techniques.


%

\section*{Acknowledgments}
The authors would like to thank the Associate Editor and the Reviewers for their valuable comments. The first author would also like to thank Shunya Yamashita and Prof.~Takeshi~Hatanaka from Tokyo Tech for their useful suggestions.

\appendices
\section*{Appendix}\label{Appendix}

\subsection{Proof of \Cref{Lemma nonlinear IFP}}
Under \Cref{Assumption strongly convex}, one has $\nabla f_i(x_i)- \nabla f_i(x_i^*)= B_{x_i}\left( x_i - x_i^*\right)$,
where 
\begin{align}
B_{x_i} = \int_{0}^{1} \nabla^2 f_i(x_i^* + \tau (x_i - x_i^*))d \tau
\end{align}
is a positive definite matrix satisfying \cite[Lemma 1]{qu2019exponential}
\begin{align}\label{eq: bounded Bxi}
\mu_i I \leq B_{x_i} \leq l_i I.
\end{align}
Clearly, $B_{x_i}$ is invertible and $B_{x_i}^{-1}$ is also positive definite.
Then, the $i$th subsystem in \eqref{subsystems} can be written as
\[
\begin{array}{rl}
\Delta \dot{x}_i &= - \alpha B_{x_i} \Delta x_i - \Delta \lambda_i + \beta u_i\\
\Delta \dot{\lambda}_i &= - \gamma u_i\\
y_i &= \Delta x_i.
\end{array}
\]
Since $\dot{x}_i^* = \dot{\lambda}_i^* \equiv \mathbf{0}$, one has $\dot{x}_i = \Delta \dot{x}_i$ and $\dot{\lambda}_i = \Delta \dot{\lambda}_i$.
Denote $z_i = \alpha \left( \nabla f_i(x_i)- \nabla f_i(x_i^*)\right) + \Delta\lambda_i$, or equivalently, 
\begin{equation}\label{z_i}
z_i = \alpha B_{x_i} \Delta x_i + \Delta \lambda_i.
\end{equation}
Let us consider the storage function
\[
\begin{array}{rl}
V_i =& \frac{\eta_i}{2} \|z_i\|^2 - \frac{1}{\gamma} \Delta x_i^T \Delta \lambda_i + \frac{\alpha}{\gamma}\left( f_i(x_i^*) - f_i(x_i) \right)\\
& + \frac{\alpha}{\gamma} \nabla f_i(x_i^*)^T \Delta x_i
\end{array}
\]
where $\eta_i$ is a positive parameter such that $\eta_i > \frac{1}{\mu_i \alpha \gamma}$.
By the strong convexity of $f_i$, one has
\[
\begin{array}{rl}
f_i(x_i^*)
\geq & f_i(x_i) - \nabla f_i(x_i)^T \Delta x_i + \frac{\mu_i}{2} \Delta x_i^T \Delta x_i\\
= & f_i(x_i) - \nabla f_i(x_i)^T \frac{B_{x_i}^{-1}}{\alpha} \left(\alpha B_{x_i} \Delta x_i\right)\\
& + \left( \alpha B_{x_i} \Delta x_i\right)^T \frac{\mu_i B_{x_i}^{-2}}{2 \alpha^2} \left( \alpha B_{x_i} \Delta x_i\right).
\end{array}
\]
Then, by substituting the above inequality into $V_i$, one gets
\begin{align*}
V_i \geq & \frac{\eta_i}{2} \|z_i\|^2 - \frac{1}{\gamma} \Delta x_i^T \Delta \lambda_i + \frac{\alpha}{\gamma} \nabla f_i(x_i^*)^T \frac{B_{x_i}^{-1}}{\alpha}\left( \alpha B_{x_i} \Delta x_i\right)\\
& - \frac{\alpha}{\gamma} \nabla f_i(x_i)^T \frac{B_{x_i}^{-1}}{\alpha} \left(\alpha B_{x_i} \Delta x_i\right)\\
& + \frac{\alpha}{\gamma}\left( \alpha B_{x_i} \Delta x_i\right)^T \frac{\mu_i B_{x_i}^{-2}}{2 \alpha^2} \left( \alpha B_{x_i} \Delta x_i\right)\\
= & \frac{\eta_i}{2} \left\| \alpha B_{x_i} \Delta x_i + \Delta \lambda_i \right\|^2 - \left( \alpha B_{x_i} \Delta x_i\right)^T \frac{B_{x_i}^{-1}}{\alpha \gamma} \Delta \lambda_i\\
& - \left( \alpha B_{x_i} \Delta x_i\right)^T \frac{B_{x_i}^{-1}}{\alpha \gamma} \left( \alpha B_{x_i} \Delta x_i\right)\\
& + \left( \alpha B_{x_i} \Delta x_i\right)^T \frac{\mu_i B_{x_i}^{-2}}{2 \alpha \gamma} \left( \alpha B_{x_i} \Delta x_i\right)\\
= & 
\begin{bmatrix}
\alpha B_{x_i} \Delta x_i\\
\Delta \lambda_i
\end{bmatrix}^T
R_i
\begin{bmatrix}
\alpha B_{x_i} \Delta x_i\\
\Delta \lambda_i
\end{bmatrix}
\end{align*}
where $R_i = \begin{bmatrix}
\frac{\eta_i}{2} I + \frac{\mu_i B_{x_i}^{-2}}{2 \alpha \gamma} -\frac{B_{x_i}^{-1}}{\alpha \gamma} & \frac{\eta_i}{2} I - \frac{B_{x_i}^{-1}}{2 \alpha \gamma}\\
* & \frac{\eta_i}{2}I
\end{bmatrix}$.
By the Schur complement \cite[Proposition 8.2.4]{bernstein2009matrix}, $R_i > 0$ if and only if $\frac{\eta_i}{2} > 0$ and
$$\frac{\eta_i}{2} I + \frac{\mu_i B_{x_i}^{-2}}{2 \alpha \gamma} -\frac{B_{x_i}^{-1}}{\alpha \gamma} - \frac{2}{\eta_i}\left(\frac{\eta_i}{2} I - \frac{B_{x_i}^{-1}}{2 \alpha \gamma} \right)^2 > 0.$$
Select $\eta_i$ such that $\eta_i > \frac{1}{\mu_i \alpha \gamma}$, then the above inequality holds, and $R_i > 0$.
Hence, $V_i > 0$ and $V_i = 0$ if and only if $(x_i,\lambda_i) = (x_i^*,\lambda_i^*)$.\\
It follows from the gradient of $f_i(x_i)$ that
\begin{equation}\label{dot-z_i}
\dot{z}_i = \alpha \nabla^2 f_i(x_i)\Delta \dot{x}_i + \Delta \dot{\lambda}_i
\end{equation}
where
\begin{equation}\label{x_i}
\Delta \dot{x}_i = \dot{x}_i = - z_i + \beta u_i.
\end{equation}
Then the derivative of $V_i$ satisfies
\begin{align*}
\dot{V}_i =& \eta_i z_i^T \left( -\alpha \nabla^2 f_i(x_i)(z_i - \beta u_i)- \gamma u_i \right)\\
& - \frac{1}{\gamma} \left( \Delta x_i^T  (-\gamma u_i) + (-z_i + \beta u_i)^T \Delta \lambda_i \right)\\
& + \frac{\alpha}{\gamma} \left\{ -\left( \nabla f_i(x_i)-\nabla f_i(x_i^*) \right)^T (-z_i + \beta u_i)\right\}\\
=& - \eta_i \alpha z_i^T \nabla^2 f_i(x_i) z_i + \eta_i z_i^T \left(\alpha \beta \nabla^2 f_i(x_i) - \gamma I \right) u_i\\
& + \Delta x_i^T u_i + \frac{1}{\gamma} z_i^T \Delta \lambda_i - \frac{\beta}{\gamma} u_i^T \Delta \lambda_i\\
& + \frac{1}{\gamma} \left( \alpha B_{x_i} \Delta x_i \right)^T z_i - \frac{\beta}{\gamma} \left( \alpha B_{x_i} \Delta x_i \right)^T u_i\\
=& - \eta_i \alpha z_i^T \nabla^2 f_i(x_i) z_i + \eta_i z_i^T \left( \alpha \beta \nabla^2 f_i(x_i) - \gamma I \right) u_i\\
& +  \Delta x_i^T u_i + \frac{1}{\gamma} \|z_i\|^2 - \frac{\beta}{\gamma} z_i^T u_i\\
\leq & - \left(\mu_i \eta_i \alpha - \frac{1}{\gamma} \right) \|z_i\|^2 + y_i^T u_i\\
&+ z_i^T \underbrace{\left\{ \eta_i \left( \alpha \beta \nabla^2 f_i(x_i) - \gamma I \right) - \frac{\beta}{\gamma}I \right\}}_{g_i} u_i\\
\leq & -\left(\mu_i \eta_i \alpha - \frac{1}{\gamma} \right) \Vert z_i \Vert^2 + \Vert z_i\Vert \Vert g_i \Vert \Vert u_i\Vert + y_i^T u_i \numberthis \label{Derivative of V_i}\\
\leq & y_i^T u_i - \nu_i u_i^T u_i
\end{align*}
where the first inequality follows from the strong convexity of $f_i$, the last inequality follows from the inequality of arithmetic and geometric means and $\eta_i > \frac{1}{\mu_i \alpha \gamma}$, and $\nu_i = - \frac{\Vert g_i \Vert^2}{4\left(\mu_i \eta_i \alpha - \frac{1}{\gamma} \right)} \leq 0$.
Since parameters in $g_i$ and $\nabla^2 f_i(x_i)$ are bounded, given finite $\eta_i$, a constant $\nu_i$ can be obtained.
Thus, the subsystem $\Sigma_i$ is IFP($\nu_i$).
\endproof


\subsection{Proof of \Cref{Theorem exponentially convergence}}


Consider the case where $\beta > 0$. Adopt the Lyapunov function candidate $V_{e} = (1 - \delta) \sum_{i \in \mathcal{N}}V_i + \frac{\delta}{2 \beta} \left\| \Delta x \right\|^2$, where $V_i$ was defined in \eqref{Storage function} and $0 < \delta < 1$ is to be decided.

Let us look at the storage function $V_i$ again.
It has been proven in the proof of \Cref{Lemma nonlinear IFP} that $$V_i \geq
\begin{bmatrix}
\alpha B_{x_i} \Delta x_i\\
\Delta \lambda_i
\end{bmatrix}^T
R_i
\begin{bmatrix}
\alpha B_{x_i} \Delta x_i\\
\Delta \lambda_i
\end{bmatrix}
> 0
$$
where $R_i = \begin{bmatrix}
\frac{\eta_i}{2} I + \frac{\mu_i B_{x_i}^{-2}}{2 \alpha \gamma} -\frac{B_{x_i}^{-1}}{\alpha \gamma} & \frac{\eta_i}{2} I - \frac{B_{x_i}^{-1}}{2 \alpha \gamma}\\
* & \frac{\eta_i}{2}I
\end{bmatrix}$. 
Denote
\begin{align*}
R_i - s I =
\begin{bmatrix}
\left( \frac{\eta_i}{2} - s \right) I + \frac{\mu_i B_{x_i}^{-2}}{2 \alpha \gamma} - \frac{B_{x_i}^{-1}}{\alpha \gamma} & \frac{\eta_i}{2} I - \frac{B_{x_i}^{-1}}{2 \alpha \gamma}\\
* & \left(\frac{\eta_i}{2} - s \right) I
\end{bmatrix}.
\end{align*}
To obtain the smallest eigenvalue of $R_i$, let us solve $\text{det}(R_i - s I) = 0$. By the Schur complement \cite[Proposition 8.2.3]{bernstein2009matrix}, it is equivalent to solving
\[
\begin{array}{ll}
& \text{det}\left( \left( \frac{\eta_i}{2}-s \right) I \right) \cdot \text{det}
 \left(
\left( \frac{\eta_i}{2} - s \right) I +
\frac{\mu_i B_{x_i}^{-2}}{2 \alpha \gamma} - \frac{B_{x_i}^{-1}}{\alpha \gamma} \right.\\
& \left.
- \left( \frac{\eta_i}{2} - s \right)^{-1} \left( \frac{\eta_i}{2} I - \frac{B_{x_i}^{-1}}{2 \alpha \gamma} \right)^2
\right) = 0\\
\Rightarrow &\text{det}\left( s^2 I - s \left( \eta_i I + \frac{\mu_i B_{x_i}^{-2}}{2 \alpha \gamma}  - \frac{B_{x_i}^{-1}}{\alpha \gamma}\right) \right.\\
& \left. + 
\frac{B_{x_i}^{-2}}{4 \alpha \gamma} \left(\mu_i \eta_i - \frac{1}{\alpha \gamma} \right) \right)= 0 \numberthis \label{eq: determinant of R_i - sI}
\end{array}
\]
where $\text{det}\left( \left( \frac{\eta_i}{2}-s \right) I \right) \neq 0$ in the above since $\frac{\eta_i}{2}$ is not an eigenvalue to $R_i$.
Notice that $ 0 \leq \frac{1}{l_i}I \leq B_{x_i}^{-1} \leq \frac{1}{\mu_i} I$ by \eqref{eq: bounded Bxi}, then there exists an invertible matrix $T_{x_i} \in \mathbb{R}^{m \times m}$ such that 
\begin{align*}
T_{x_i}^{-1} B_{x_i}^{-1} T_{x_i} = \Lambda_{x_i}
\end{align*}
where $\Lambda_{x_i} \in \mathbb{R}^{m \times m}$ is a diagonal matrix.
Then \eqref{eq: determinant of R_i - sI} becomes
\begin{equation}\label{eq: solve quadratic equation}
\begin{aligned}
& \prod_{j = 1}^{m}\left[ s^2 - s \left( \eta_i + \frac{\mu_i r_j^2}{2 \alpha \gamma} - \frac{r_j}{\alpha \gamma} \right) + \frac{r_j^2}{4 \alpha \gamma} \left( \mu_i \eta_i - \frac{1}{\alpha \gamma} \right) \right]= 0
\end{aligned}
\end{equation}
where $\frac{1}{l_i} \leq r_j \leq \frac{1}{\mu_i} $ is the $j$th diagonal element of $\Lambda_{x_i}$. Since $\eta_i + \frac{\mu_i r_j^2}{2 \alpha \gamma} - \frac{r_j}{\alpha \gamma} > 0$ and $\mu_i \eta_i - \frac{1}{\alpha \gamma} >0$ by \Cref{Lemma nonlinear IFP}, the roots $s_{j}^{\pm}$, $\forall j$ to \eqref{eq: solve quadratic equation} are positive.
Solving \eqref{eq: solve quadratic equation}, we obtain
\[
\begin{array}{rl}
s_j^{-} = & \frac{1}{2} \left( \vphantom{- \sqrt{\left( \eta_i + \frac{\mu_i r_j^2}{2\alpha \gamma} - \right)^2}} \left(\eta_i + \frac{\mu_i r_j^2}{2\alpha \gamma} - \frac{r_j}{\alpha \gamma} \right) \right.\\
& \left. - \sqrt{\left( \eta_i + \frac{\mu_i r_j^2}{2\alpha \gamma} - \frac{r_j}{\alpha \gamma} \right)^2 - \frac{r_j^2}{\alpha \gamma} \left( \mu_i \eta_i - \frac{1}{\alpha \gamma}\right)} \right).
\end{array}
\]
Denote $s_j^{-}(r_j)$ as a function of $r_j$ with an abuse of notation. The smallest eigenvalue of $R_i$ satisfies
\begin{align}\label{eq: smallest eigenvalue of R_i}
s_1(R_i) \geq \varepsilon_i := \min_{\frac{1}{l_i} \leq r_j \leq \frac{1}{\mu_i}} \{s_{j}^{-}(r_j) \},
\end{align}
then 
$
V_i \geq \varepsilon_i \left\|
\begin{smallmatrix}
\alpha B_{x_i} \Delta x_i\\
\Delta \lambda_i
\end{smallmatrix}
\right\|^2.
$

Next, let us derive the upper bound of $V_i$. By the strong convexity,
\begin{align*}
f_{i}(x_{i}^{*}) - f_{i}(x_{i})\leq -\nabla f_i(x_{i}^{*})^{T}\Delta x_{i}-\frac{\mu_i}{2}\left\| \Delta x_{i}\right\|^{2}.
\end{align*}
Substituting the above inequality into $V_i$, we have
\[
\begin{array}{rl}
V_i \leq & \frac{\eta_i}{2} \| z_i \|^2 - \frac{1}{\gamma} \Delta x_i^T \Delta \lambda_i + \frac{\alpha}{\gamma}\left( -\frac{\mu_i}{2} \left\| \Delta x_{i}\right\|^{2} \right)\\
\leq & \frac{\eta_i}{2} \left\| \alpha B_{x_i} \Delta x_i + \Delta \lambda_i \right\|^2
- \frac{1}{\gamma} \left(\alpha B_{x_i} \Delta x_i \right)^T \frac{B_{x_i}^{-1}}{\alpha} \Delta \lambda_i\\
& - \left(\alpha B_{x_i} \Delta x_i \right)^T \frac{\mu_i B_{x_i}^{-2}}{2 \alpha \gamma} \left(\alpha B_{x_i} \Delta x_i \right)\\
= &
\begin{bmatrix}
\alpha B_{x_i} \Delta x_i\\
\Delta \lambda_i
\end{bmatrix}^T
M_i
\begin{bmatrix}
\alpha B_{x_i} \Delta x_i\\
\Delta \lambda_i
\end{bmatrix}
\end{array}
\]
where
$M_i =
\begin{bmatrix}
\frac{\eta_i}{2} I - \frac{\mu_i B_{x_i}^{-2}}{2 \alpha \gamma} & \frac{\eta_i}{2} I - \frac{B_{x_i}^{-1}}{2 \alpha \gamma}\\
* & \frac{\eta_i}{2}I
\end{bmatrix}$.
Denote the matrix
\begin{align*}
\eta_i I - M_i = 
\begin{bmatrix}
\frac{\eta_i}{2} I + \frac{\mu_i B_{x_i}^{-2}}{2 \alpha \gamma} & -\frac{\eta_i}{2} I + \frac{B_{x_i}^{-1}}{2 \alpha \gamma}\\
* & \frac{\eta_i}{2}I
\end{bmatrix}.
\end{align*}
By similar application of the Schur complement \cite[Proposition 8.2.4]{bernstein2009matrix}, we can obtain that $\eta_i I - M_i \geq 0$.

Let $\eta_i = \frac{2}{\mu_i \alpha \gamma}$, satisfying \Cref{Lemma nonlinear IFP}, then $M_i \leq \frac{2}{\mu_i \alpha \gamma} I$. Moreover,
$s_j^{-}(r_j)$ is monotonically increasing with respect to $r_j$. Thus, \eqref{eq: smallest eigenvalue of R_i} leads to
\[
\begin{array}{rl}
\varepsilon_i = s_j^{-} \left(\frac{1}{l_i} \right)
= & \frac{1}{2 \mu_i \alpha \gamma} \left(
\vphantom{\sqrt{ \left( \left(\frac{\mu_i}{l_i} \right)^2 \right)^2}}
2 + \frac{1}{2} \left(\frac{\mu_i}{l_i} \right)^2 - \frac{\mu_i}{l_i} - \right.\\
& \left. \sqrt{ \left( 2 + \frac{1}{2} \left(\frac{\mu_i}{l_i} \right)^2 - \frac{\mu_i}{l_i} \right)^2  - \left(\frac{\mu_i}{l_i} \right)^2} \right).
\end{array}
\] 

Denote $\mu = \min_{i \in \mathcal{N}} \{\mu_i\}$, $l = \max_{i \in \mathcal{N}} \{l_i\}$ and $B_{x} = \text{diag} \left\{ B_{x_1}, \ldots, B_{x_N} \right\}$.
It satisfies that $\nabla f(x) - \nabla f(x^*) = B_{x} \Delta x$ and 
$0 \leq \left\|\Delta x \right\|^2 \leq \frac{1}{\alpha^2 \mu^2} \left\|\alpha B_{x} \Delta x \right\|^2$ due to \eqref{eq: bounded Bxi}.
By the definition of $V_{e}$, we obtain 
\[
\begin{array}{ll}
& (1 - \delta) \min_{i \in \mathcal{N}} \{\varepsilon_i \} 
\left\|
\begin{smallmatrix}
\alpha B_{x} \Delta x\\
\Delta \lambda
\end{smallmatrix}
\right\|^2
\leq V_{e}\\
& \leq \left( (1 - \delta) \frac{2}{\alpha \gamma \mu} + \frac{\delta}{2 \beta \alpha^2 \mu^2} \right)
\left\|
\begin{smallmatrix}
\alpha B_{x} \Delta x\\
\Delta \lambda
\end{smallmatrix}
\right\|^2. \numberthis \label{eq: lower upper bounds of Ve}
\end{array}
\]
Let $\eta_i = \frac{2}{\mu_i \alpha \gamma}$, then $\nu_i = - \frac{ \gamma}{4}\max_{x_i}\{\|g_i\|^2\}$ by \eqref{problem to get nu}, where $g_i$ was defined in \eqref{Derivative of V_i}. Since $\nu_i \leq 0$ by \Cref{Lemma nonlinear IFP}, let us assume without loss of generality that $\max_{i}\{ d^{i} |\nu_i|\} \neq 0$. Choose a constant $\rho \in \left( 0, \frac{\delta \max_{i \in \mathcal{N}} \{ |\nu_i| d^{i} \}}{\max_{i \in \mathcal{N}} \{ d^{i} \}}  \right)$, where $d^{i}$ is the in/out-degree, and then denote 
\begin{align}\label{eq: definition of phi}
\phi = \frac{1 - \delta}{\gamma} - \frac{1 - \delta}{\gamma + \frac{4 \rho}{(1 - \delta) \max_{x_i} \{ \| g_i \|^2 \}}} > 0.
\end{align}
Substituting $\eta_i = \frac{2}{\mu_i \alpha \gamma}$ into \eqref{Derivative of V_i}, the time derivative of $(1 - \delta)V$ satisfies 
\[
\begin{array}{ll}
& (1 - \delta)\dot{V} \\
\leq & - \left( \frac{1 - \delta}{\gamma} \right) \displaystyle \sum_{i \in \mathcal{N}} \| z_i \|^2 + (1 - \delta)\sum_{ i\in\mathcal{N} } \| z_i \| \| g_i \| \| u_i\|\\
& + (1 - \delta)\Delta x^T u\\
= & - \displaystyle \sum_{i \in \mathcal{N}} \left( \frac{1 - \delta}{\gamma} - \frac{1 - \delta}{\gamma + \frac{4 \rho}{(1 - \delta)\| g_i \|^2 }} \right) \| z_i \|^2 + (1 - \delta)\Delta x^T u\\
& - \displaystyle \sum_{i \in \mathcal{N}}\frac{1 - \delta}{\gamma + \frac{4 \rho}{(1 - \delta)\| g_i \|^2 }} \| z_i \|^2 + (1 - \delta) \sum_{i \in \mathcal{N}} \| z_i \| \| g_i \| \| u_i \| \\
\leq & - \phi \| z \|^2
+ \displaystyle \sum_{i \in \mathcal{N}} \frac{(1 - \delta)^2 \| g_i \|^2}{4 (1 - \delta)} \left( \gamma + \frac{ 4 \rho}{(1 - \delta) \| g_i \|^2 } \right) \| u_i \|^2\\
& + (1 - \delta) \Delta x^T u\\
\leq & - \phi \| z \|^2 + \displaystyle \sum_{i \in \mathcal{N}} \left( -(1 - \delta)\nu_i + \rho \right) \| u_i \|^2 + (1 - \delta) \Delta x^T u
\end{array}
\]
where $z = \text{col} (z_1, \ldots, z_N)$, the second inequality follows from \eqref{eq: definition of phi} and the inequality of arithmetic and geometric means similarly to \eqref{Derivative of V_i}, and the last inequality follows from the definition of $\nu_i$. The above manipulation provides a term containing $\|z\|^2$ in order to prove negative semi-definiteness of $\dot{V}_{e}$ later.

Therefore, the time derivative of $V_{e}$ satisfies
\[
\begin{array}{rl}
\dot{V}_e = & (1- \delta) \dot{V} + \frac{\delta}{\beta} \Delta x^T \Delta \dot{x}\\
\leq & - \phi \| z \|^2 + \displaystyle \sum_{i \in \mathcal{N}} \left( -(1 - \delta)\nu_i + \rho \right) \| u_i \|^2 + \Delta x^T u\\
& - \frac{\delta \alpha}{\beta} \Delta x^T B_{x}\Delta x - \frac{\delta}{\beta} \Delta x^T \Delta \lambda.
\end{array}
\]
Replacing $\nu_i$ by $\left( (1 - \delta)\nu_i - \rho \right)$ in \eqref{Derivative of V}, we have
\begin{align*}
& \displaystyle \sum_{i \in \mathcal{N}} \left( -(1 - \delta)\nu_i + \rho \right) \| u_i \|^2 + \Delta x^T u\\
\leq & - \sigma(t) \displaystyle \sum_{i \in \mathcal{N}} \left( \frac{1}{2} - \sigma(t) \left( \rho - (1 - \delta)\nu_i \right) d^{i} \right) \cdot\\
& \displaystyle \sum_{j\in \mathcal{N}_i} a_{ij} \left \| \Delta x_j - \Delta x_i \right \|^2\\
\leq & - \left( \frac{\sigma(t)}{2} - \sigma^2(t) \left( \rho \max_{i \in \mathcal{N}} \{d^{i}\} + (1 - \delta) \max_{i \in \mathcal{N}} \{ |\nu_i| d^{i}\}\right) \right) \cdot\\
& \displaystyle \sum_{i \in \mathcal{N}} \sum_{j\in \mathcal{N}_i} a_{ij} \left\lVert \Delta x_j - \Delta x_i \right\rVert^2\\
\leq & - \varphi \displaystyle \sum_{i \in \mathcal{N}} \sum_{j\in \mathcal{N}_i} a_{ij} \left\lVert \Delta x_j - \Delta x_i \right\rVert^2 = - \varphi \Delta x^T \mathbf{L} \Delta x
\end{align*}
where 
\begin{align*}
& \varphi = \min_{t} \left\{ \frac{\sigma(t)}{2} - \sigma^2(t) ( \rho \max_{i \in \mathcal{N}} \{d^{i}\} + (1 - \delta) \max_{i \in \mathcal{N}} \{ |\nu_i| d^{i}\} )
\right\} \numberthis
\end{align*}
and $\varphi > 0$ by the definition of $\rho$ and condition \eqref{coupling gain switching graphs 2}.

Let us define $\bar{x}$ as the stacked vector of the average value of $x_i$, i.e., $\bar{x} : = \mathbf{1}_{N} \otimes \left(\frac{1}{N} \left(\mathbf{1}_{N} \otimes I_{m}\right)^T x \right)$.
Observe that for any vector $v \in \mathbb{R}^{m}$, $\left(\mathbf{1}_{N} \otimes v \right)^T \left( x - \bar{x} \right) = 0$.
In addition, $\left(\mathbf{1}_{N} \otimes v \right)^T (\mathbf{L}^T + \mathbf{L}) =  \mathbf{1}_{N}^T (L^T + L) \otimes v  = \mathbf{0}$, which implies that $\left( x - \bar{x} \right)$ is orthogonal to all eigenvectors of $\mathbf{L}^T + \mathbf{L}$ associated with zero eigenvalues.
Consequently, we have 
\begin{align}\label{eq: s2L < xLx}
\Delta x^T \mathbf{L} \Delta x = 
\left( x - \bar{x} \right)^T \mathbf{L} \left( x - \bar{x} \right) \geq s_{2} \left\| x - \bar{x} \right\|^2
\end{align}
where the equality follows from $\mathbf{1}_{N}^T {L} =\mathbf{0}$, ${L} \mathbf{1}_{N} = \mathbf{0}$, and $s_{2} : = s_{2}\left( L + L^T \right)$ is the smallest nonzero eigenvalue of $L + L^T$.\\
We can also observe that
\[
\begin{array}{ll}
& \Delta x^T \Delta \lambda\\
= & (x - \bar{x})^T \Delta \lambda + 
\bar{x}^T \Delta \lambda - {x^*}^T \Delta \lambda\\
= & (x - \bar{x})^T \Delta \lambda - {x_i^*}^T (\mathbf{1}_{N} \otimes I_m)^T \Delta \lambda\\
& + \left(\frac{1}{N} \left(\mathbf{1}_{N} \otimes I_{m}\right)^T x \right)^T (\mathbf{1}_{N} \otimes I_m)^T \Delta \lambda\\
= & (x - \bar{x})^T \Delta \lambda
 \numberthis \label{eq: bar_x times lambda equals zero}
\end{array}
\]
where the second equality follows from the Kronecker product and the last equality is due to \eqref{eq: sum of lambda equals zero}.
Consequently,
\begin{align*}
\dot{V}_e
\leq & - \phi \| z \|^2 - \frac{\delta \alpha}{\beta} \Delta x^T B_{x}\Delta x - \varphi s_{2} \left\|x - \bar{x} \right\|^2\\
& - \frac{\delta}{\beta} \left(x - \bar{x}\right)^T \Delta \lambda\\
\leq & - \phi \| z \|^2 - \frac{\delta \alpha}{\beta}\Delta x^T B_{x}\Delta x - \varphi s_{2} \left\|x - \bar{x} \right\|^2\\
    & + \frac{\delta}{\beta} \left( \frac{\theta}{2} \left\|x - \bar{x} \right\|^2 + \frac{1}{2\theta} \left\|\Delta \lambda\right\|^2 \right)\\
= & 
    - \begin{bmatrix}
    \alpha B_{x} \Delta x \\ \Delta \lambda
    \end{bmatrix}^T
    \begin{bmatrix}
    \phi I + \frac{\delta}{\alpha \beta} B_{x}^{-1} & \phi I\\
    * & \phi I - \frac{\delta}{2 \beta \theta}I
\end{bmatrix}
\begin{bmatrix}
    \alpha B_{x} \Delta x \\ \Delta \lambda
\end{bmatrix}\\
& - \left( \varphi s_2 - \frac{\delta \theta}{2\beta} \right) \left\|x - \bar{x} \right\|^2\\
\leq & 
    - \begin{bmatrix}
    \alpha B_{x} \Delta x \\ \Delta \lambda
    \end{bmatrix}^T
    \underbrace{
    \begin{bmatrix}
    \phi I + \frac{\delta}{\alpha \beta l}I & \phi I\\
    * & \phi I - \frac{\delta}{2 \beta \theta}I
\end{bmatrix}
    }_{Q}
\begin{bmatrix}
\alpha B_{x} \Delta x \\ \Delta \lambda
\end{bmatrix}\\
& - \left( \varphi s_2 - \frac{\delta\theta}{2 \beta} \right) \left\|x - \bar{x} \right\|^2
\end{align*}
where the first inequality follows from \eqref{eq: s2L < xLx}, \eqref{eq: bar_x times lambda equals zero}, the second inequality follows from the Young's inequality with $\theta > 0$, the equality follows from \eqref{z_i}, and the last inequality follows from \eqref{eq: bounded Bxi}.
Observe that $\dot{V}_e$ is negative definite if $Q > 0$ and $\left( \varphi s_2 - \frac{\delta\theta}{2 \beta} \right) > 0$, i.e., the following conditions hold,
\[
\begin{array}{ll}
    2 \beta \varphi s_2 - \delta \theta > 0, ~~
    2 \theta - \alpha l > 0,\\
    2 \beta \phi \theta - \delta > 0,~~ (2 \theta - \alpha l) \beta \phi - \delta > 0.
\end{array}
\]
Choose $\theta = \alpha l$, then the above conditions become
$\delta < \alpha \beta l \phi $ and $\delta < \frac{2 \beta \varphi s_2}{\alpha l}$.
Though $\phi (\delta)$, $\varphi (\delta)$ are functions of $\delta$, it is obvious that $\phi (\delta), \varphi (\delta) > 0$ when $\delta \rightarrow 0$.
Then there always exists a small enough $\delta \in \left(0, \min\left\{ \alpha \beta l \phi, \frac{2 \beta \varphi s_2}{\alpha l}\right\} \right)$ such that the above conditions are satisfied and $\dot{V_{e}}$ is negative definite.

Next, by calculations,
$
    \dot{V_{e}} \leq - s_{1}(Q)
    \left\|
    \begin{smallmatrix}
    \alpha B_{x} \Delta x \\ \Delta \lambda
    \end{smallmatrix} \right\|^2
$,
where $s_1(Q) = \frac{2\phi + \frac{\delta}{2\alpha \beta l} - \sqrt{4\phi^2 + \frac{9\delta^2}{4\alpha^2 \beta^2 l^2} } }{2} > 0$ is the smallest eigenvalue of $Q$.
Then, by the exponential stability theorem \cite[Theorem 4.10]{khalil1996noninear}, 
we have $V_{e}(t) \leq V_{e}(0) e^{-\epsilon t}$,
where
\begin{equation*}
\epsilon = \frac{s_{1}(Q)}{(1 - \delta) \frac{2}{\alpha \gamma \mu} + \frac{\delta}{2 \beta \alpha^2 \mu^2}} = \frac{2\phi + \frac{\delta}{2\alpha \beta l} - \sqrt{4\phi^2 + \frac{9\delta^2}{4\alpha^2 \beta^2 l^2} } }{(1 - \delta) \frac{4}{\alpha \gamma \mu} + \frac{\delta}{\beta \alpha^2 \mu^2}}
\end{equation*}
due to \eqref{eq: lower upper bounds of Ve}, and
\begin{align*}
& \left\|
\begin{smallmatrix}
    \alpha B_{x} \Delta x \\ \Delta \lambda
    \end{smallmatrix} \right\|
\leq 
\left( \frac{ (1 - \delta) \frac{2}{\alpha \gamma \mu} + \frac{\delta}{2 \beta \alpha^2 \mu^2} }{(1 - \delta) \min_{i \in \mathcal{N}} \{\varepsilon_i \} } \right)^{\frac{1}{2}}
\left\| \begin{smallmatrix}
    \alpha B_{x} \Delta x (0) \\ \Delta \lambda (0)
    \end{smallmatrix} \right\| e^{- \frac{\epsilon t}{2}}.
\end{align*}
Recall that $\left\|\Delta x \right\| \leq \frac{1}{\alpha \mu} \left\|\alpha B_{x}\Delta x \right\| \leq \frac{l}{\mu} \left\|\Delta x \right\|$ and $B_{x}\Delta x = \mathbf{0}$ if and only if $\Delta x = \mathbf{0}$ due to \eqref{eq: bounded Bxi}.
Then we finally obtain 
\begin{align*}
&\left\|
\begin{smallmatrix}
\Delta x \\
\Delta \lambda
\end{smallmatrix}
\right\|
\leq 
\frac{l}{\mu} \left( \frac{ (1 - \delta) \frac{2}{\alpha \gamma \mu} + \frac{\delta}{2 \beta \alpha^2 \mu^2} }{(1 - \delta) \min_{i \in \mathcal{N}} \{\varepsilon_i \} } \right)^{\frac{1}{2}} \left\|
\begin{smallmatrix}
\Delta x(0) \\
\Delta \lambda(0)
\end{smallmatrix}
\right\| e^{-\frac{\epsilon t}{2}}
\end{align*}
for any $t \geq 0$.

The cases for $\beta < 0$ and $\beta = 0$ can be proved similarly by taking $V_{e} = (1+ \delta) V - \frac{\delta}{2 \beta} \left\| \Delta x \right\|^2$ and $\dot{V}_{e} = V + \frac{\delta}{2} \left\| \Delta x \right\|^2$, respectively.


\subsection{Proof of \Cref{Lemma The matrix is nonsingular}}

To prove the nonsingularity, we first give some lemmas as follows.
\begin{lemma}\label{Lemma invertible}
Given a real matrix $Q \in \mathbb{R}^{m \times m}$, the matrix $(I + Q)$ is invertible if and only if $-1$ is not an eigenvalue of $Q$.
\end{lemma}
\begin{proof}
$(I+Q)$ is invertible if and only if $\det(I+Q) \neq 0$. In other words, the characteristic polynomial $p_{-Q}(s)$ at $1$ is nonzero, i.e., $p_{-Q}(1) = \det(I+Q) \neq 0$. Since eigenvalues are the roots of the characteristic polynomial, it means that $1$ is not an eigenvalue of $-Q$, or equivalent, $-1$ is not an eigenvalue of $Q$.
\end{proof}
\begin{lemma}\label{Lemma non-negative eigenvalues of products}
Let $L \in \mathbb{R}^{N \times N}$ be the Laplacian matrix for a weight-balanced graph $\mathcal{G}$, $M \in \mathbb{R}^{N \times N}$ be a diagonal matrix and $M \geq 0$, then the eigenvalues of $M L$ or $L M$ have non-negative real parts.
\end{lemma}
\begin{proof}
Let us consider the case of $M L$. Since $M$ is diagonal, by direct calculation,
\begin{equation}
ML = 
\begin{bmatrix}
M_{11}L_{11} & \ldots & M_{11} L_{1N}\\
\vdots & \ddots & \vdots\\
M_{NN} L_{N1} & \ldots & M_{NN} L_{NN}
\end{bmatrix}.
\end{equation}
Since $M \geq 0$ and $\mathcal{G}$ is weight-balanced, we have $ M_{ii}L_{ii} = \sum_{j \neq i} \left| M_{ii} L_{ij} \right |$ for all $i$, implying that $ML$ is diagonal dominant. By the Gershgorin circle theorem \cite[Fact 4.10.16]{bernstein2009matrix},  the real parts of the eigenvalues of $ML$ remain non-negative. The case of $LM$ can be proved similarly.
\end{proof}
We are now ready to prove the nonsingularity.

Take $m = 1$ without loss of generality. Recall that $\nu_i \leq 0$, $\forall i \in \mathcal{N}$ by \Cref{Lemma nonlinear IFP}. The eigenvalues of $- \mathbf{L}(t)\nu$ have non-negative real parts by \Cref{Lemma non-negative eigenvalues of products}. In addition, $\sigma(t) > 0$, thus $-1$ is not an eigenvalue of $- \sigma(t) \mathbf{L}(t) \nu $. By \Cref{Lemma invertible}, $\left(I - \sigma(t) \mathbf{L}(t) \nu \right)$ is invertible and hence nonsingular.


\ifCLASSOPTIONcaptionsoff
 \newpage
\fi



%

\bibliographystyle{IEEEtran}
\bibliography{References}

\begin{thebibliography}{10}
\providecommand{\url}[1]{#1}
\csname url@samestyle\endcsname
\providecommand{\newblock}{\relax}
\providecommand{\bibinfo}[2]{#2}
\providecommand{\BIBentrySTDinterwordspacing}{\spaceskip=0pt\relax}
\providecommand{\BIBentryALTinterwordstretchfactor}{4}
\providecommand{\BIBentryALTinterwordspacing}{\spaceskip=\fontdimen2\font plus
\BIBentryALTinterwordstretchfactor\fontdimen3\font minus
  \fontdimen4\font\relax}
\providecommand{\BIBforeignlanguage}[2]{{%
\expandafter\ifx\csname l@#1\endcsname\relax
\typeout{** WARNING: IEEEtran.bst: No hyphenation pattern has been}%
\typeout{** loaded for the language `#1'. Using the pattern for}%
\typeout{** the default language instead.}%
\else
\language=\csname l@#1\endcsname
\fi
#2}}
\providecommand{\BIBdecl}{\relax}
\BIBdecl

\bibitem{li2019input}
M.~Li, G.~Chesi, and Y.~Hong, ``Input-feedforward-passivity-based distributed
  optimization over directed and switching topologies,'' in \emph{2019 IEEE
  58th Conference on Decision and Control (CDC)}.\hskip 1em plus 0.5em minus
  0.4em\relax IEEE, 2019, pp. 6056--6061.

\bibitem{nedic2015distributed}
A.~Nedi{\'c} and A.~Olshevsky, ``Distributed optimization over time-varying
  directed graphs,'' \emph{IEEE Transactions on Automatic Control}, vol.~60,
  no.~3, pp. 601--615, 2015.

\bibitem{nedic2017achieving}
A.~Nedi{\'c}, A.~Olshevsky, and W.~Shi, ``Achieving geometric convergence for
  distributed optimization over time-varying graphs,'' \emph{SIAM Journal on
  Optimization}, vol.~27, no.~4, pp. 2597--2633, 2017.

\bibitem{xie2018distributed}
P.~Xie, K.~You, R.~Tempo, S.~Song, and C.~Wu, ``Distributed convex optimization
  with inequality constraints over time-varying unbalanced digraphs,''
  \emph{IEEE Transactions on Automatic Control}, vol.~63, no.~12, pp.
  4331--4337, 2018.

\bibitem{wang2010control}
J.~Wang and N.~Elia, ``Control approach to distributed optimization,'' in
  \emph{2010 48th Annual Allerton Conference on Communication, Control, and
  Computing (Allerton)}.\hskip 1em plus 0.5em minus 0.4em\relax IEEE, 2010, pp.
  557--561.

\bibitem{yi2015distributed}
P.~Yi, Y.~Hong, and F.~Liu, ``Distributed gradient algorithm for constrained
  optimization with application to load sharing in power systems,''
  \emph{Systems \& Control Letters}, vol.~83, pp. 45--52, 2015.

\bibitem{li2018generalized}
M.~Li, ``Generalized lagrange multiplier method and {KKT} conditions with an
  application to distributed optimization,'' \emph{IEEE Transactions on
  Circuits and Systems II: Express Briefs}, vol.~66, no.~2, pp. 252--256, 2019.

\bibitem{zeng2018distributedb}
X.~Zeng, P.~Yi, Y.~Hong, and L.~Xie, ``Distributed continuous-time algorithms
  for nonsmooth extended monotropic optimization problems,'' \emph{SIAM Journal
  on Control and Optimization}, vol.~56, no.~6, pp. 3973--3993, 2018.

\bibitem{forti2004generalized}
M.~Forti, P.~Nistri, and M.~Quincampoix, ``Generalized neural network for
  nonsmooth nonlinear programming problems,'' \emph{IEEE Transactions on
  Circuits and Systems I: Regular Papers}, vol.~51, no.~9, pp. 1741--1754,
  2004.

\bibitem{martinez2007motion}
S.~Martinez, J.~Cortes, and F.~Bullo, ``Motion coordination with distributed
  information,'' \emph{IEEE Control Systems Magazine}, vol.~27, no.~4, pp.
  75--88, 2007.

\bibitem{wang2015distributed}
X.~Wang, Y.~Hong, and H.~Ji, ``Distributed optimization for a class of
  nonlinear multiagent systems with disturbance rejection,'' \emph{IEEE
  transactions on Cybernetics}, vol.~46, no.~7, pp. 1655--1666, 2015.

\bibitem{zhang2017distributed}
Y.~Zhang, Z.~Deng, and Y.~Hong, ``Distributed optimal coordination for multiple
  heterogeneous euler--lagrangian systems,'' \emph{Automatica}, vol.~79, pp.
  207--213, 2017.

\bibitem{hatanaka2018passivity}
T.~Hatanaka, N.~Chopra, T.~Ishizaki, and N.~Li, ``Passivity-based distributed
  optimization with communication delays using {PI} consensus algorithm,''
  \emph{IEEE Transactions on Automatic Control}, vol.~63, no.~12, pp.
  4421--4428, 2018.

\bibitem{kia2015distributed}
S.~S. Kia, J.~Cort{\'e}s, and S.~Mart{\'\i}nez, ``Distributed convex
  optimization via continuous-time coordination algorithms with discrete-time
  communication,'' \emph{Automatica}, vol.~55, pp. 254--264, 2015.

\bibitem{gharesifard2014distributed}
B.~Gharesifard and J.~Cort{\'e}s, ``Distributed continuous-time convex
  optimization on weight-balanced digraphs,'' \emph{IEEE Transactions on
  Automatic Control}, vol.~59, no.~3, pp. 781--786, 2014.

\bibitem{li2017distributed}
Z.~Li, Z.~Ding, J.~Sun, and Z.~Li, ``Distributed adaptive convex optimization
  on directed graphs via continuous-time algorithms,'' \emph{IEEE Transactions
  on Automatic Control}, vol.~63, no.~5, pp. 1434--1441, 2017.

\bibitem{deng2017distributed}
Z.~Deng, S.~Liang, and Y.~Hong, ``Distributed continuous-time algorithms for
  resource allocation problems over weight-balanced digraphs,'' \emph{IEEE
  Transactions on Cybernetics}, vol.~48, no.~11, pp. 3116--3125, 2018.

\bibitem{zhu2018continuous}
Y.~Zhu, W.~Yu, G.~Wen, and W.~Ren, ``Continuous-time coordination algorithm for
  distributed convex optimization over weight-unbalanced directed networks,''
  \emph{IEEE Transactions on Circuits and Systems II: Express Briefs}, 2018.

\bibitem{chopra2006passivity}
N.~Chopra and M.~W. Spong, ``Passivity-based control of multi-agent systems,''
  in \emph{Advances in robot control}.\hskip 1em plus 0.5em minus 0.4em\relax
  Springer, 2006, pp. 107--134.

\bibitem{chopra2012output}
N.~Chopra, ``Output synchronization on strongly connected graphs,'' \emph{IEEE
  Transactions on Automatic Control}, vol.~57, no.~11, pp. 2896--2901, 2012.

\bibitem{liu2015output}
T.~Liu, D.~J. Hill, and J.~Zhao, ``Output synchronization of dynamical networks
  with incrementally-dissipative nodes and switching topology,'' \emph{IEEE
  Transactions on Circuits and Systems I: Regular Papers}, vol.~62, no.~9, pp.
  2312--2323, 2015.

\bibitem{stegink2017unifying}
T.~Stegink, C.~De~Persis, and A.~van~der Schaft, ``A unifying energy-based
  approach to stability of power grids with market dynamics,'' \emph{IEEE
  Transactions on Automatic Control}, vol.~62, no.~6, pp. 2612--2622, 2017.

\bibitem{tang2017distributed}
Y.~Tang and P.~Yi, ``Distributed coordination for a class of non-linear
  multi-agent systems with regulation constraints,'' \emph{IET Control Theory
  \& Applications}, vol.~12, no.~1, pp. 1--9, 2017.

\bibitem{yamashita2020passivity}
S.~Yamashita, T.~Hatanaka, J.~Yamauchi, and M.~Fujita, ``Passivity-based
  generalization of primal--dual dynamics for non-strictly convex cost
  functions,'' \emph{Automatica}, vol. 112, p. 108712, 2020.

\bibitem{qu2014modularized}
Z.~Qu and M.~A. Simaan, ``Modularized design for cooperative control and
  plug-and-play operation of networked heterogeneous systems,''
  \emph{Automatica}, vol.~50, no.~9, pp. 2405--2414, 2014.

\bibitem{proskurnikov2017simple}
A.~V. Proskurnikov and M.~Mazo~Jr, ``Simple synchronization protocols for
  heterogeneous networks: beyond passivity,'' \emph{IFAC-PapersOnLine},
  vol.~50, no.~1, pp. 9426--9431, 2017.

\bibitem{li2019consensus}
M.~Li, L.~Su, and G.~Chesi, ``Consensus of heterogeneous multi-agent systems
  with diffusive couplings via passivity indices,'' \emph{IEEE Control Systems
  Letters}, vol.~3, no.~2, pp. 434--439, 2019.

\bibitem{zhang2018distributed}
J.~Zhang, K.~You, and T.~Ba{\c{s}}ar, ``Distributed discrete-time optimization
  in multiagent networks using only sign of relative state,'' \emph{IEEE
  Transactions on Automatic Control}, vol.~64, no.~6, pp. 2352--2367, 2018.

\bibitem{khalil1996noninear}
H.~K. Khalil, ``Nonlinear systems,'' \emph{Prentice-Hall, New Jersey}, 1996.

\bibitem{sepulchre2012constructive}
R.~Sepulchre, M.~Jankovic, and P.~V. Kokotovic, \emph{Constructive nonlinear
  control}.\hskip 1em plus 0.5em minus 0.4em\relax Springer Science \& Business
  Media, 2012.

\bibitem{qu2019exponential}
G.~Qu and N.~Li, ``On the exponential stability of primal-dual gradient
  dynamics,'' \emph{IEEE Control Systems Letters}, vol.~3, no.~1, pp. 43--48,
  2019.

\bibitem{boyd2004convex}
S.~Boyd and L.~Vandenberghe, \emph{Convex optimization}.\hskip 1em plus 0.5em
  minus 0.4em\relax Cambridge university press, 2004.

\bibitem{kottenstette2014relationships}
N.~Kottenstette, M.~J. McCourt, M.~Xia, V.~Gupta, and P.~J. Antsaklis, ``On
  relationships among passivity, positive realness, and dissipativity in linear
  systems,'' \emph{Automatica}, vol.~50, no.~4, pp. 1003--1016, 2014.

\bibitem{barkana2014defending}
I.~Barkana, ``Defending the beauty of the invariance principle,''
  \emph{International Journal of Control}, vol.~87, no.~1, pp. 186--206, 2014.

\bibitem{franceschelli2013decentralized}
M.~Franceschelli, A.~Gasparri, A.~Giua, and C.~Seatzu, ``Decentralized
  estimation of laplacian eigenvalues in multi-agent systems,''
  \emph{Automatica}, vol.~49, no.~4, pp. 1031--1036, 2013.

\bibitem{charalambous2015distributed}
T.~Charalambous, M.~G. Rabbat, M.~Johansson, and C.~N. Hadjicostis,
  ``Distributed finite-time computation of digraph parameters:
  Left-eigenvector, out-degree and spectrum,'' \emph{IEEE Transactions on
  Control of Network Systems}, vol.~3, no.~2, pp. 137--148, 2015.

\bibitem{antipin1994feedback}
A.~Antipin, ``Feedback-controlled saddle gradient processes,'' \emph{Automation
  and Remote Control}, vol.~55, no.~3, pp. 311--320, 1994.

\bibitem{zeng2018distributeda}
X.~Zeng, S.~Liang, Y.~Hong, and J.~Chen, ``Distributed computation of linear
  matrix equations: an optimization perspective,'' \emph{IEEE Transactions on
  Automatic Control}, no.~5, pp. 1858--1873, 2019.

\bibitem{astrom2010feedback}
K.~J. Astr{\"o}m and R.~M. Murray, \emph{Feedback systems: an introduction for
  scientists and engineers}.\hskip 1em plus 0.5em minus 0.4em\relax Princeton
  university press, 2010.

\bibitem{simpson2018equilibrium}
J.~W. Simpson-Porco, ``Equilibrium-independent dissipativity with quadratic
  supply rates,'' \emph{IEEE Transactions on Automatic Control}, no.~4, pp.
  1440--1455, 2019.

\bibitem{lofberg2004yalmip}
J.~L{\"o}fberg, ``Yalmip: A toolbox for modeling and optimization in matlab,''
  in \emph{Proceedings of the CACSD Conference}, vol.~3.\hskip 1em plus 0.5em
  minus 0.4em\relax Taipei, Taiwan, 2004.

\bibitem{li2021distributed}
M.~Li, L.~Su, and T.~Liu, ``Distributed optimization with event-triggered
  communication via input feedforward passivity,'' \emph{IEEE Control Systems
  Letters}, vol.~5, no.~1, pp. 283--288, 2021.

\bibitem{bernstein2009matrix}
D.~S. Bernstein, \emph{Matrix mathematics: theory, facts, and formulas}.\hskip
  1em plus 0.5em minus 0.4em\relax Princeton university press, 2009.

\end{thebibliography}

%

\begin{IEEEbiography}[{\includegraphics[width=1in,height=1.25in,clip,keepaspectratio]{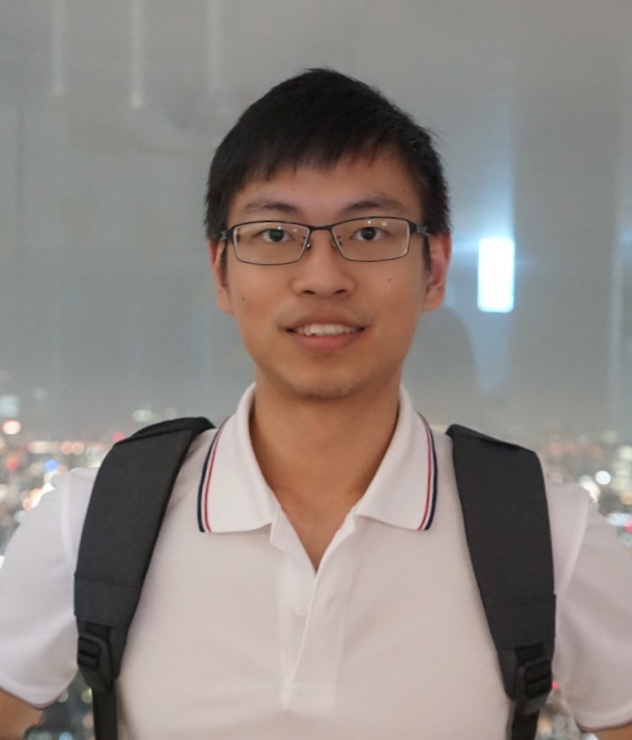}}]{Mengmou Li} received his B.S. degree in Physics from Zhejiang University, China, in 2016, and the Ph.D. degree in Electrical and Electronic Engineering from the University of Hong Kong, in 2020. From October 2018 to December 2018, he was a visiting student in Academy of Mathematics and Systems Science, Chinese Academy of Sciences (CAS), China. From September 2019 to December 2019, he was a visiting researcher in School of Engineering, Osaka University, Japan. His research interests include distributed optimization, consensus and synchronization, passivity.
\end{IEEEbiography}

\begin{IEEEbiography}[{\includegraphics[width=1in,height=1.25in,clip,keepaspectratio]{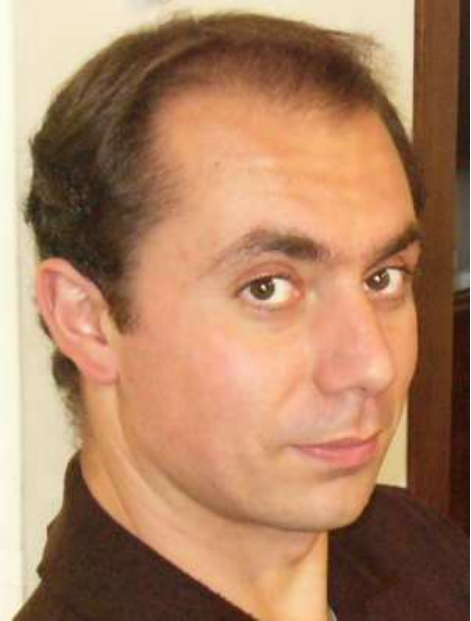}}]{Graziano Chesi} is a Professor at the Department of Electrical and Electronic Engineering of the University of Hong Kong. He received the Laurea in Information Engineering (summa cum laude and encomium) and the Best Student Award of the Faculty of Engineering from the University of Florence in 1997, and the PhD in Systems Engineering from the University of Bologna in 2001. He joined the University of Siena in 2000, and the University of Hong Kong in 2006. He served as Associate Editor for various journals, including Automatica, the European Journal of Control, the IEEE Control Systems Letters, the IEEE Transactions on Automatic Control, the IEEE Transactions on Computational Biology and Bioinformatics, and Systems and Control Letters. He also served as Guest Editor for the IEEE Transactions on Automatic Control, for the International Journal of Robust and Nonlinear Control, and for Mechatronics. He founded and served as chair of the Technical Committee on Systems with Uncertainty of the IEEE Control Systems Society. He also served as chair of the Best Student Paper Award Committees for the IEEE Conference on Decision and Control and the IEEE Multi-Conference on Systems and Control. He is author of the books ``Homogeneous Polynomial Forms for Robustness Analysis of Uncertain Systems'' (Springer 2009) and ``Domain of Attraction: Analysis and Control via SOS Programming'' (Springer 2011). He is a Fellow of the IEEE.
\end{IEEEbiography}

\begin{IEEEbiography}[{\includegraphics[width=1in,height=1.25in,clip,keepaspectratio]{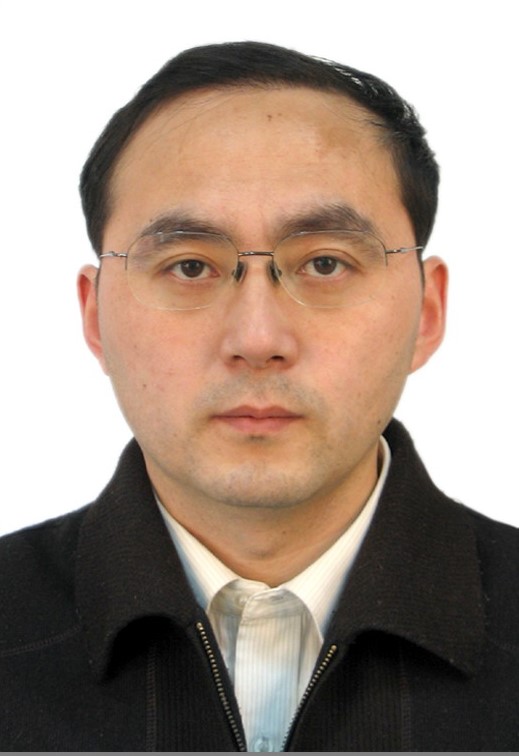}}]{Yiguang Hong} received his B.S. and M.S. degrees from Dept of Mechanics of Peking University, China, and the Ph.D. degree from the Chinese Academy of Sciences (CAS), China. He is currently a professor in Academy of Mathematics and Systems Science, CAS. Also, he is a Fellow of IEEE, a Fellow of Chinese Association for Artificial Intelligence, and a Fellow of Chinese Association of Automation (CAA). Additionally, he is the chair of Technical Committee of Control Theory of CAA and was a board of governor of IEEE Control Systems Society.
His current research interests include nonlinear control, multi-agent systems, distributed optimization and game, machine learning, and social networks. He serves as Editor-in-Chief of Control Theory and Technology. He also serves or served as Associate Editors for many journals including the IEEE Transactions on Automatic Control, IEEE Transactions on Control of Network Systems, and IEEE Control Systems Magazine. Moreover, he is a recipient of the Guang Zhaozhi Award at the Chinese Control Conference, Young Author Prize of the IFAC World Congress, Young Scientist Award of CAS, the Youth Award for Science and Technology of China, and the National Natural Science Prize of China.
\end{IEEEbiography}



\end{document}